\documentclass[11pt,reqno]{mcom-l}
\usepackage{amsthm,amssymb,amsfonts,amsmath,epsfig,hyperref}
\usepackage{mathabx}
\usepackage[margin=1.1in]{geometry}
\usepackage{mathrsfs}
\usepackage{graphics}
\usepackage{cite}            
\usepackage{xcolor}

\newtheorem{theorem}{Theorem}[section]
\newtheorem{lemma}[theorem]{Lemma}

\newtheorem{proposition}[theorem]{Proposition}
\theoremstyle{definition}

\newtheorem{assumption}{Assumption}[section]
\theoremstyle{remark}
\newtheorem{remark}{\bf Remark}[section]

\definecolor{darkred}{rgb}{.7,0,0}

\definecolor{green}{rgb}{0,0.7,0}

\renewcommand{\theequation}{\thesection.\arabic{equation}}
\allowdisplaybreaks

\def\R{\mathbb{R}}

\def\d{\,{\rm d}}
\def\le{\leqslant}
\def\ge{\geqslant}
\def\Omega{\varOmega}
\def\Lambda{\varLambda}
\def\sh{\mathring S_h}
\def\uh{\check{u}_h}
\def\vh{\check{v}_h}
\def\xh{\check{\chi}_h}

\def\K{{\mathscr K}}

\title{
Weak discrete maximum principle of isoparametric\\ finite element methods 
in curvilinear polyhedra}

\author[Buyang Li]{Buyang Li}
\address{Buyang Li. Department of Applied Mathematics,
The Hong Kong Polytechnic University, Hung Hom, Hong Kong.}
\email {\href{mailto:buyang.li@polyu.edu.hk}{buyang.li{\it @}polyu.edu.hk}}

\author[Weifeng Qiu]{Weifeng Qiu}
\address{Weifeng Qiu. Department of Mathematics,
City University of Hong Kong, Hung Hom, Hong Kong.}
\email {\href{mailto:weifeqiu@cityu.edu.hk}{weifeqiu@cityu.edu.hk}}

\author[Yupei Xie]{Yupei Xie}
\address{Yupei Xie. 
\indent Department of Applied Mathematics,
The Hong Kong Polytechnic University, Hung Hom, Hong Kong.}
\email {\href{mailto:yupei.xie@polyu.edu.hk}{yupei.xie@polyu.edu.hk}}

\author[Wenshan Yu]{Wenshan Yu}
\address{Wenshan Yu. Division of Science and Technology, United International College (BNU-HKBU), Zhuhai, 519087, P.R. China.}
\email {\href{mailto:yuwenshan@uic.edu.cn}{yuwenshan@uic.edu.cn}}

\thanks{This work is partially supported by the Research Grants Council of the Hong Kong Special Administrative Region, China (GRF Project No. PolyU15300519, CityU 11302219), and an internal grant at The Hong Kong Polytechnic University (Project ID: P0038843, Work Programme: ZVX7).
}

\begin{document}

\maketitle

\begin{abstract}
\small


The weak maximum principle of the isoparametric finite element method is proved for the Poisson equation under the Dirichlet boundary condition in a (possibly concave) curvilinear polyhedral domain with edge openings smaller than $\pi$, which include smooth domains and smooth deformations of convex polyhedra. The proof relies on the analysis of a dual elliptic problem with a discontinuous coefficient matrix arising from the isoparametric finite elements. Therefore, the standard $H^2$ elliptic regularity which is required in the proof of the weak maximum principle in the literature does not hold for this dual problem. To overcome this difficulty, we have decomposed the solution into a smooth part and a nonsmooth part, and estimated the two parts by $H^2$ and $W^{1,p}$ estimates, respectively. 

As an application of the weak maximum principle, we have proved a maximum-norm best approximation property of the isoparametric finite element method for the Poisson equation in a curvilinear polyhedron. The proof contains non-trivial modifications of Schatz's argument due to the non-conformity of the iso-parametric finite elements, which requires us to construct a globally smooth flow map which maps the curvilinear polyhedron to a perturbed larger domain on which we can establish the $W^{1,\infty}$ regularity estimate of the Poisson equation uniformly with respect to the perturbation.

\end{abstract}

\pagestyle{myheadings}
\markboth{}{}

\section{Introduction}

Let $\Omega$ be a bounded domain in $\R^N$ with $N\in\{2,3\}$ and consider a quasi-uniform triangulation of the domain $\Omega$ with mesh size $h$, denoted by $\K_h$. 
Hence, $\Omega_h = (\bigcup_{K\in\K_h}K)^\circ$ is an approximation of $\Omega$. 
Let $S_h(\Omega_h)$ be a finite element space subject to the triangulation $\K_h$, and denote by $\mathring S_h(\Omega_h)=\{v_h\in S_h(\Omega_h):v_h=0\,\,\mbox{on}\,\,\partial\Omega_h\}$ the finite element subspace under the homogeneous boundary condition. 
A function $u_h\in S_h(\Omega_h)$ is called discrete harmonic if it satisfies
\begin{align}\label{discrete-harmonic}
\int_{\Omega_h}\nabla u_h\cdot \nabla \chi_h=0\quad\forall\,\chi_h\in\mathring S_h(\Omega_h).
\end{align}
For a given mesh and finite element space, if all the discrete harmonic functions satisfy the following inequality: 
\begin{align}\label{max-principle}
\left\|u_{h}\right\|_{L^{\infty}(\Omega_h)} \leqslant \left\|u_{h}\right\|_{L^{\infty}(\partial \Omega_h)} , 
\end{align}
then it is said that the {\it discrete maximum principle} holds.

The discrete maximum principle of finite element methods (FEMs) has attracted much attention from numerical analysts due to its importance for the stability and accuracy of numerical solutions; for example, see \cite{Ciarlet_1970, CiarletRaviart_1973, Ruas Santos_1982, Vanselow_2001, WangZhang_2012}. However, strong restrictions on the geometry of the mesh are required for the discrete maximum principle to hold. 
For example, for piecewise linear finite elements on a two-dimensional triangular mesh, the discrete maximum principle generally requires the angles of the triangles to be less than $\pi/2$; see \cite[\textsection 5]{WangZhang_2012}.
In three dimensions, it is hard to have such discrete maximum principe even for piecewise linear finite elements; see \cite{Brandts_Korotov_Krizek_2009, Korotov_Krizek_2001, Korotov_Krizek_Pekka_2001, Xu_Zikatanov_1999}.

Schatz considered a different approach in \cite{Schatz-1980} by proving the weak maximum principle (also called the Agmon--Miranda maximum principle) ,
\begin{align}\label{weak-max-principle}
\left\|u_{h}\right\|_{L^{\infty}(\Omega_h)} \leqslant C\left\|u_{h}\right\|_{L^{\infty}(\partial \Omega_h)} ,
\end{align} 
for some constant $C$ which is independent of $u_h$ and $h$, 
for a wide class of $H^1$-conforming finite elements on a general quasi-uniform triangulation of a two-dimensional polygonal domain. It was shown in \cite{Schatz-1980} that the weak maximum principle can be used to prove the maximum-norm stability and best approximation results of FEMs in a plane polygonal domain, i.e.,
\begin{equation}\label{Rhu-Linfty}
\|u-R_hu\|_{L^\infty(\Omega)} \le C\ell_h \inf_{v_h\in \mathring S_h}\|u-v_h\|_{L^\infty(\Omega)}
\quad\forall\, u\in H^1_0(\Omega)\cap L^\infty(\Omega) ,
\end{equation}  
where $R_h:H^1_0(\Omega)\rightarrow \mathring S_h$ is the Ritz projection operator, and 
$$
\ell_h=
\left\{\begin{aligned}
&\ln(2+1/h) &&\mbox{for piecewise linear elements},\\
&\,1 &&\mbox{for higher-order finite elements}.
\end{aligned}\right.
$$ 
Such maximum-norm stability and best approximation results have a number of applications in the error estimates of finite element solutions for parabolic problems \cite{2019-Li, 2021-Li, Kashiwabara-Kemmochi-1, Leykekhman_Vexler_2016b},  Stokes systems \cite{Behringer_Leykekhman_Vexler_2019}, nonlinear problems \cite{Frehse_Rannacher_1978, Demlow_2006, Meinder_Vexler_2016b}, optimal control problems \cite{Apel_Rosch_Sirch_2009, Apel_Winkler_Pfefferer_2018}, and so on. 


In three dimensions, the weak maximum principle was extended to convex polyhedral domains in \cite{Leykekhman_Li_2021} and used to prove the $L^\infty$-norm stability and best approximation results of FEMs on convex polyhedral domains, removing an extra logarithmic factor $\ln(2+1/h)$ in the stability constant for quadratic and higher-order elements obtained in other approaches (for example, see \cite{Leykekhman_Vexler_2016}). 
When $\Omega$ is a smooth domain and $\Omega_h=\Omega$ (the triangulation is assumed to match the curved boundary exactly), the weak maximum principle of quadratic or higher-order FEMs is a result of the maximum-norm stability result in \cite{Schatz-Wahlbin-1977, Schatz-Wahlbin-1982}, and the weak maximum principle of linear finite elements can be proved similarly as in \cite{Leykekhman_Li_2021}. In all these articles, the triangulation is assumed to match the boundary of the domain exactly, with $\Omega_h=\Omega$.

%
%

In the practical computation, the curved boundary of a bounded smooth domain, or more generally a curvilinear polygon or polyhedron (which may contain both curved faces, curved edges, and corners), is generally approximated by isoparametric finite elements instead of being matched exactly by the triangulation. In this case, the weak maximum principle of FEMs has not been proved yet. Correspondingly, the best approximation results such as \eqref{Rhu-Linfty} are not known for isoparametric FEMs in a curved domain. 

Some related results have been proved in the case $\Omega_h\neq\Omega$. For the Poisson equation with Dirichlet boundary conditions in convex smooth domains, the piecewise linear finite element space with a zero extension in $\Omega\backslash \Omega_h$ is conforming, i.e., $S_{h}(\Omega_h) \subset H_{0}^{1}(\Omega)$. In this case, pointwise error estimates of FEMs have been established in \cite{Bakaev-Thomee-Wahlbin-2002,Schatz-Wahlbin-1982}. 
For general bounded smooth domains which may be concave, thus the finite element space may be non-conforming, Kashiwabara \& Kemmochi \cite{Kashiwabara-Kemmochi-2} have obtained the following error estimate for piecewise linear finite elements for the Poisson equation under the Neumann boundary condition: 
\begin{align}\label{Error-KK}
\|\tilde{u}-u_{h}\|_{L^{\infty}(\Omega_{h})} 
&\le C h|\log h| \inf_{v_h\in S_h} \|\tilde{u}-v_{h}\|_{W^{1, \infty}(\Omega_{h})}+C h^{2}|\log h|\|u\|_{W^{2, \infty}(\Omega)} , 
\end{align} 
where $\tilde u$ is any extension of $u$ in $W^{2,\infty}(\Omega_\delta)$ and $\Omega_\delta$ is a neighborhood of $\overline\Omega$. 
In the case $u\in W^{2,\infty}(\Omega)$, this error estimate is a consequence of the best approximation result in \eqref{Rhu-Linfty}. More recently, the $W^{1,\infty}$ stability of the Ritz projection was proved in \cite{dorich_23} for isoparametric FEMs on $C^{r+1,1}$-smooth domains based on weighted-norm estimates, where $r$ denotes the degree of finite elements. For curvilinear polyhedra or smooth domains which may be concave, the weak maximum principle and the best approximation results in the $L^\infty$ norm have not been proved. 
 
In this article, we close the gap mentioned above by proving the weak maximum principle in \eqref{weak-max-principle} for isoparametric finite elements of degree $r\ge 1$ in a bounded smooth domain or a curvilinear polyhedron (possibly concave) with edge openings smaller than $\pi$.
As an application of the weak maximum principle, we prove that the finite element solution $u_h\in\mathring S_h(\Omega_h)$ of the Poisson equation 
\begin{align}\label{Poisson}
\left\{
\begin{aligned}
-\Delta u &= f&&\mbox{in}\,\,\,\Omega \\
u &= 0 &&\mbox{on}\,\,\, \partial\Omega 
\end{aligned}\right. 
\end{align}
using isoparametric finite elements of degree $r\ge 1$ has the following optimal-order error bound (for any $p>N$):
\begin{align}\label{error-estimate}
 \|u-u_h\|_{L^\infty(\Omega)}
 \le C\ell_h\|u-\check{I}_hu\|_{L^\infty(\Omega)} + Ch^{r+1}\ell_h\|f\|_{L^p(\Omega)} , 
\end{align}
where $u_h$ is extended to be zero in $\Omega\backslash\Omega_h$, and $\check{I}_hu$ denotes a Lagrange interpolation operator (which will be defined in the next section). 
Inequality \eqref{error-estimate} can be viewed as a variant of the best approximation result in \eqref{Rhu-Linfty} by taking account of the geometry change of the domain, which produces an additional optimal-order term $Ch^{r+1}\|f\|_{L^p(\Omega)} $ independent of the higher regularity of $f$. 
In particular, inequality \eqref{error-estimate} implies the following error estimate: 
\begin{align}\label{error-estimate2}
 \|u-u_h\|_{L^\infty(\Omega)}
 \le C\ell_h h^s \|u\|_{C^s(\overline\Omega)} + Ch^{r+1}\ell_h\|f\|_{L^p(\Omega)} 
\quad \mbox{for}\,\,\, u\in C^s(\overline\Omega),\,\,\, 0\le s\le r+1,
\end{align}
which adapts to the regularity of $u$. 

The weak maximum principle is proved by converting the finite element weak form on $\Omega_h$ to a weak form on $\Omega$ by using a bijective transformation $\Phi_h:\Omega_h\rightarrow\Omega$ which is piecewisely defined on the triangles/tetrahedra. This yields a bilinear form with a discontinuous coefficient matrix. 
The main technical difficulty is that the elliptic partial differential equation associated to this coefficient matrix does not have the $H^2$ regularity estimate, which is required in the proof of weak maximum principle in the literature; see \cite{Leykekhman_Li_2021}. 
We overcome this difficulty by decomposing the finite element solution $v_h$ (of a duality problem) into two parts, 
$v_h=v_{h,1}+v_{h,2}$, with $v_{h,1}$ corresponding to the Poisson equation with $H^2$ regularity, and $v_{h,2}$ corresponding to an elliptic equation with discontinuous coefficients but with a small source term arising from the geometry perturbation, 
and then estimate the two parts separately by using the $H^2$ and $W^{1,p}$ regularity of the respective problems. 

The maximum-norm error estimate is proved by using Schatz argument through estimating the difference between the solutions of the Poisson equations in $\Omega_h$ and $\Omega$. However, in order to avoid using the partial derivatives of $f$ in the proof of \eqref{error-estimate}, we have to estimate the error between the solutions of the Poisson equation in the two domains $\Omega_h$ and $\Omega$ under the Dirichlet boundary conditions, respectively. This is accomplished by perturbing the curvilinear polyhedron through a globally smooth flow map pointing outward the domain and establishing the $W^{1,\infty}$ regularity estimate of the Poisson equation in a slightly larger perturbed domain $\Omega^t$ (uniformly with respect to the perturbation), which contains both $\Omega_h$ and $\Omega$ and satisfies that ${\rm dist}(x,\partial\Omega) \sim h^{r+1}$ for $x\in\partial\Omega^t$.  

The rest of this article is organized as follows. 
In Section \ref{sec:isopara}, we present the main results to be proved in this article, including the weak maximum principle of the isoparametric FEM in a curvilinear polyhedron, and the best approximation result of finite element solutions in the maximum norm. The proofs of the two main results are presented in Sections \ref{section:3} and \ref{sec: stability Ritz}, respectively. The conclusions are presented in Section \ref{section:conclusion}. 


\section{Main results}\label{sec:isopara}
\setcounter{equation}{0}


%

In this article, we assume that $\Omega\subset\mathbb{R}^N$, with $N\in\{2,3\}$, is either a bounded smooth domain or a curvilinear polyhedron (possibly concave) with edge openings smaller than $\pi$. 
More specifically, in the three-dimensional space, this means that for every $x\in \partial\Omega$ there is a neighborhood $U_x$ and a smooth diffeomorphism $\varphi_{x}:U_x\to B_0(\varepsilon_x)$ mapping $x$ to $0$ such that one of the following three conditions holds: 
\begin{enumerate}
	\item $x$ is a smooth point, i.e.,  $\varphi_{x}(U_x\cap \Omega)= B_0(\varepsilon_x)\cap \R^3_+$, where $\R^3_+=\{x\in\R^3:x_3>0\}$ is a half space in $\R^3$.
	\item $x$ is an edge point, i.e., $\varphi_{x}(U_x\cap \Omega)= B_0(\varepsilon_x)\cap K_x$, where $K_x=\R\times\Sigma$, where $\Sigma\subseteq\mathbb{R}^2$ is a sector with angle less than $\pi$.
	\item 
	$x$ is a vertex point, i.e., $\varphi_{x}(U_x\cap \Omega)= B_0(\varepsilon_x)\cap K_x ,$ 
	where $K_x$ is a convex polyhedral cone with a vertex at $0$. Therefore, the boundary of $K_x$ consists of several smooth faces intersecting at some edges which pass through the vertex $0$. 

\end{enumerate}
We refer to \cite[Definition 2.1]{Ludmi} for the definition of general curvilinear polyhedron. 

%
%
%

Let $\K$ be the set of closed simplices in a quasi-uniform triangulation of the domain $\Omega$ with isoparametric finite elements of degree $r\ge 1$ approximating the boundary $\partial\Omega$, as described in \cite{Lenoir_1986}, with flat interior simplices which have at most one vertex on $\partial\Omega$ and possibly curved boundary simplices. Each boundary simplex contains a possibly curved face or edge interpolating $\partial\Omega$ with an accuracy of $O(h^{r+1})$, where $h$ denotes the mesh size of the triangulation. Hence, $\Omega_h=(\bigcup_{K\in\K}K)^\circ$ is an approximation to $\Omega$ such that ${\rm dist}(x,\partial \Omega)=O(h^{r+1})$ for $x\in \partial\Omega_{h}$.  

We prove the following weak maximum principle of the isoparametric FEM.

\begin{theorem}\label{THM1}
For the isoparametric FEM of degree $r\ge 1$ on a quasi-uniform triangulation of $\Omega$, all the discrete harmonic functions $u_h\in S_h(\Omega_h)$ satisfying \eqref{discrete-harmonic} have the following estimate: 
\begin{align}\label{weak-max-principle}
\left\|u_{h}\right\|_{L^{\infty}(\Omega_h)} \leqslant C\left\|u_{h}\right\|_{L^{\infty}(\partial \Omega_h)} ,
\end{align}
where the constant $C$ is independent of $u_h$ and the mesh size $h$.
\end{theorem}

In the isoparametric finite elements described in \cite{Lenoir_1986}, each curved simplex $K\in \K$ is the image of a map $F_{K}:\hat K\to K$ defined on the reference simplex $\hat K$, which is a polynomial of degree no larger than $r$ and transforms the finite element structure of $\hat K$ to $K$. There is a homeomorphism $\Phi_h:\Omega_h\to \Omega$, which is piecewise smooth on each simplex and globally Lipschitz continuous. If we denote $\Phi_{h,K}:=\Phi_h|_K$ and $\check{K}:=\Phi_h(K)$, then $\Phi_{h,K}:K\rightarrow \check{K}$ is a diffeomorphism which transforms the finite element structure of $K$ to $\check{K}$. Therefore, $\check{\K}=\{\check{K}:K\in\K\}$ is a triangulation of the curved domain $\Omega$, with 
$$
\Omega_h=\bigcup_{K\in\K}K
\quad\mbox{and}\quad
\Omega=\bigcup_{K\in\K}\check{K}. 
$$ 
One can define isoparametric finite element space $S_h(\Omega_{h})$ as
\begin{equation}\label{fem-def1}
	S_h(\Omega_{h})=\{v_h\in H^1(\Omega_{h}): v_h|_K\circ F_K \mbox{ is a polynomial on $\hat K$ of degree\,$\le r\,$ for $K\in \mathscr{K}$}\} . 
\end{equation}
The finite element spaces on $\Omega$ can be defined as 
\begin{equation}\label{fem-def2}
	S_h(\Omega)=\{\check{v}_h\in H^1(\Omega): \check{v}_h\circ\Phi_h \in S_h(\Omega_h) \}
	\quad\mbox{and}\quad
	\mathring S_h(\Omega)=\{ \check{v}_h\in  S_h(\Omega): \check{v}_h=0\,\,\mbox{on}\,\,\partial\Omega \} . 
\end{equation}

For a finite element function $v_h\in S_h(\Omega_h)$, we can associate it with a finite element function $\check{v}_h\in  S_h(\Omega)$ defined by $v_h\circ\Phi_h^{-1} :=\check{v}_h$. 


\begin{remark}
{\it 
By using the notation which link $v_h\in S_h(\Omega_h)$ and $\check{v}_h\in S_h(\Omega)$, the weak maximum principle in \eqref{weak-max-principle} can be equivalently written as 
\begin{align}
\left\|\check{u}_h\right\|_{L^\infty(\Omega)}\leqslant C\left\|\check{u}_h\right\|_{L^\infty(\partial \Omega)} .
\end{align}
}
\end{remark}

For a function $f\in C^0(\overline\Omega_h)$, one can define its local interpolation $I_{h,K}f$ on a simplex $K\in\K$ as the function satisfying 
$$
( I_{h,K}f )\circ F_K := I_{\hat K}(f\circ F_K) ,
$$
where $I_{\hat K}$ is the standard Lagrange interpolation on the reference simplex $\hat K$ (onto the space of polynomials of degree$\,\le r$). The global interpolation $I_hf\in S_h(\Omega_h)$ is defined as 
$$
I_hf |_{K}:= I_{h,K}f \quad\forall\, K\in\K. 
$$
For the analysis of the isoparametric FEM, we also define an interpolation operator $\check{I}_h: C(\overline\Omega)\rightarrow  S_h(\Omega)$ by 
$$
(\check{I}_h v) \circ\Phi_h = I_h(v\circ\Phi_h) \quad\forall\, v\in C(\overline\Omega). 
$$

As an application of the weak maximum principle, we establish an $L^\infty$-norm best approximation result of isoparametric FEM for the Poisson equation in a curvilinear polyehdron. 
We assume that the triangulation can be extended to a bigger domain which contains $\overline\Omega$, as stated below. 

\begin{assumption}\label{Assumption}
The curvilinear polyhedral domain $\Omega$ can be extended to a larger convex polyhedron $\Omega_*$ with piecewise flat boundaries such that $\overline\Omega\subset \Omega_*$ and the triangulation $\K$ can be extended to a quasi-uniform triangulation $\K_*$ on $\Omega_*$
(thus the triangulation in $\Omega_*\backslash\overline\Omega$ is also isoparametric on its boundary $\partial\Omega$). 
\end{assumption}
\begin{remark}
	Here $\Omega_*$ can be chosen as a large cube whose interior contains $\overline{\Omega}$. Note that the triangulation $\mathscr{K}$ is obtained from some triangulation $\widetilde{\mathscr{K}}$ consisting of flat simplexes by the method in Lenoir's paper \cite{Lenoir_1986}. We can first extend $\widetilde{K}$ to a quasi-uniform flat triangulation $\widetilde{\mathscr{K}}_*$ of $\Omega_*$, and then modify those flat simplexes $\widetilde{\mathscr{K}}$ with one of whose edges/faces attaches to the boundary $\partial\Omega$, to isoparametric elements by the method in Lenoir's paper \cite{Lenoir_1986}. This leads to a quasi-uniform triangulation $\mathscr{K}_*$ on $\Omega_*$ which extends $\mathscr{K}$. By our construction, the triangulation on $\Omega_*\setminus\overline{\Omega}$ is also isoparametric on its boundary $\partial\Omega$.
\end{remark}
\begin{theorem}\label{THM2}
For $f\in L^p(\Omega)$ with some $p>N$, we consider the Poisson equation 
\begin{align}\label{Poisson_Eq}
\left\{\begin{aligned}
-\Delta u &= f &&\mbox{in}\,\,\,\Omega\\
u&=0 &&\mbox{on}\,\,\,\partial\Omega\\
\end{aligned}\right.
\end{align}
and the isoparametric FEM of degree $r\ge 1$ for \eqref{Poisson_Eq}{\rm:} Find $u_h\in \mathring S_h(\Omega_h)$ such that  
\begin{align}\label{FEM_u_h}
\int_{\Omega_h}\nabla u_h\cdot \nabla \chi_h \d x =\int_{\Omega_h} \tilde f \chi_h \d x 
\quad\forall\,\chi_h\in\mathring S_h(\Omega_h) , 
\end{align}
where $\tilde f\in L^p(\Omega\cup\Omega_h)$ is any extension of $f\in L^p(\Omega)$ satisfying 
$
\|\tilde f\|_{L^p(\Omega\cup\Omega_h)}
\le C\| f\|_{L^p(\Omega)} .
$ 
Assuming that the triangulation satisfies Assumption {\rm\ref{Assumption}}, there exist positive constants $h_0$ and $C$ {\rm(}independent of $f$, $u$ and $h)$ such that the solutions of \eqref{Poisson_Eq} and \eqref{FEM_u_h} satisfy the following inequality for $h\leq h_0${\rm:} 
\begin{align}
    \|u-u_h\|_{L^\infty(\Omega)}\leq C \ell_h\|u-\check{I}_hu\|_{L^\infty(\Omega)}+Ch^{r+1}\ell_h\|f\|_{L^p(\Omega)} , 
\end{align}
where $u_h$ is extended to be zero on $\Omega\backslash\Omega_h$, and $\ell_h$ is defined as 
$$
\ell_h=
\left\{\begin{aligned}
	&\ln(2+1/h) &&\mbox{for piecewise linear elements},\\
	&\,1 &&\mbox{for higher-order finite elements}.
\end{aligned}\right.
$$ 

\end{theorem}


The proofs of Theorems \ref{THM1} and \ref{THM2} are presented in the next two sections, respectively. For the simplicity of notation, we denote by $C$ a generic positive constant which may be different at different occurrences, possibly depending on the specific domain $\Omega$ and the shape-regularity and quasi-uniformity of the triangulation, and the polynomial degree $r\ge 1$,
but is independent of the mesh size $h$.

\section{Proof of Theorem \ref{THM1}}\label{section:3}
\setcounter{equation}{0}

The proof of Theorem \ref{THM1} is divided into six parts, presented in the following six subsections. 

\subsection{Properties of the isomparametric FEM}

In this subsection, we summarize the basic properties of the isoparametric FEM to be used in the proof of Theorem \ref{THM1}. 

\begin{lemma}[\!\!{\cite[Theorem 1, Theorem 2, Proposition 2, Proposition 3, Proposition 4]{Lenoir_1986}}] \label{Lemma:basic}
Let $\check{\K}$ be the triangulation of $\Omega$ by isoparametric finite elements of degree $r\ge 1$, with the maps $F_K:\hat K\rightarrow K$ and $\Phi_{h,K}:K\rightarrow\check{K}$ described in Section {\rm\ref{sec:isopara}}. 
Let $D^s$ denote the Fr\'echet derivative of order $s$. Then the following results hold{\rm:} 
\begin{enumerate}
    \item[1.] $F_K:\hat K\to K$ is a diffeomorphism such that 
    \begin{align}\label{F_K-property}
        \begin{aligned}
        &\|D^sF_K\|_{L^\infty(\hat K)}\leq Ch^s &&\forall s\in [1,r+1] \\
        &\|D^sF^{-1}_K\|_{L^\infty(K)}\leq Ch^{-s}&&\forall s\in [1,r+1]
        \end{aligned}
    \end{align}
    \item[2.]  $\Phi_{h,K}:K\to \check{K}$ is a diffeomorphism such that 
    \begin{align}
        \begin{aligned}
        &\|D^s(\Phi_{h,K}-{\rm Id})\|_{L^\infty(K)}\leq Ch^{r+1-s} &&\forall s\in [1,r+1]\\
        &\|D^s(\Phi_{h,K}^{-1}-{\rm Id})\|_{L^\infty(\check{K})}\leq Ch^{r+1-s} &&\forall s\in [1,r+1]
        \end{aligned}
    \end{align}
    \item[3.]  
    For $v\in H^m(K)$ and integer $m\in [0,r+1]$, the norms $\|v\|_{H^m(K)}$ and $\|v\circ \Phi^{-1}_{h,K}\|_{H^m(\check{K})}$ are uniformly equivalent with respect to $h$. 
    \item[4.]  
    Each curved simplex $K\in\K$ corresponds to a flat simplex $\widetilde{K}$ 
    (which has the same vertices as $K$), and there is a unique linear bijection $F_{\widetilde{K}}:\hat K\to \widetilde{K}$ which maps the reference simplex $\hat K$ onto $\widetilde{K}$. The map $\widetilde{\Psi}_K:= F_K\circ F^{-1}_{\widetilde{K}}: \widetilde{K}\to K$ is a diffeomorphism satisfying the following estimates: 
    \begin{align}\label{map-Psi_K}
        \begin{aligned}
        &\|D(\widetilde{\Psi}_K-{\rm Id})\|_{L^\infty(\widetilde{K})}\leq Ch, \quad \|D(\widetilde{\Psi}^{-1}_K-{\rm Id})\|_{L^\infty(K)}\leq Ch \\
        &\|D^s\widetilde{\Psi}_K\|_{L^\infty(\widetilde{K})}\leq C, \quad \|D^s\widetilde{\Psi}^{-1}_K\|_{L^\infty(K)}\leq C\quad\forall s\in [1,r+1] . 
        \end{aligned}
    \end{align} 
    \item[5.]  For $v\in H^m(K)$ and integer $m\in [0,r+1]$, the norms $\|v\|_{H^m(K)}$ and $\|v\circ \widetilde{\Psi}_K\|_{H^m(\widetilde{K})}$ are uniformly equivalent with respect to $h$. 
\end{enumerate}
\end{lemma}

Let $W^{k,p}_h(\Omega)$ be the space of functions on $\Omega$ whose restriction on each $\check{K}\in \check{\K}$ lies in $W^{k,p}(\check{K})$, equipped with the following norm:
$$
\|v\|_{W^{k,p}_h(\Omega)}
:=
\left\{\begin{aligned}
&\bigg(\sum_{K\in\K} \|v\|_{W^{k,p}(K)}^p \bigg)^{\frac1p}
&&\mbox{for}\,\,\,1\le p<\infty ,\\
&\sup_{K\in\K} \|v\|_{W^{k,p}(K)} 
&&\mbox{for}\,\,\, p=\infty .\\
\end{aligned}\right. 
$$ 
In the case $p=2$ we write $H^{l,h}(\Omega)= W^{l,2}_h(\Omega)$. The following local interpolation error estimate was proved in \cite[Lemma 7]{Lenoir_1986}; also see \cite[Theorem 4.3.4]{Ciarlet_2002}. 
Although it was proved only for $p=2$ in \cite[Lemma 7]{Lenoir_1986}, the proof can be extended to $1\le p\le \infty$ straightforwardly. 

\begin{lemma}[Lagrange interpolation]\label{Proposition-A0}
Let $\check{I}_{h,K}:C(\overline\Omega)\rightarrow S_h(\Omega)$ be the interpolation operator defined by 
$$
\check{I}_{h,K}f \circ \Phi_h := I_{h,K}f \quad\forall\, f\in C(\overline\Omega) . 
$$
Then, for $1\leq k\leq r+1$ and $1\le p\le \infty$ such that $W^{k,p}_h(\Omega)\hookrightarrow C(\overline\Omega)$ {\rm(}e.g., $kp>N$ when $p>1$ or $k\geq N$ when $p=1${\rm)}, 
the following error estimate holds: 
$$
|u-\check{I}_{h,K}u|_{W^{i,p}(\check{K})}\leq Ch^{k-i}\|u\|_{W^{k,p}(\check{K})} \quad \forall\, 0\leq i\leq k, 
\,\, \forall\, \check{K}\in \check{\K}, 
\,\, \forall\, u\in C(\widebar{\Omega})\cap W^{k,p}_h(\Omega) .
$$
\end{lemma}

Since the Lagrange interpolation is defined by using the pointwise values of a function at the Lagrange nodes, its stability in the $W^{k,p}$ norm is valid only when $W^{k,p}(\Omega)\hookrightarrow C(\overline\Omega)$, i.e., in the case ``$kp>N$ and $p>1$'' or ``$k\geq N$ and $p=1$''. 
One can remove this restriction by using the Scott--Zhang interpolation, which can be constructed first in the flat triangulation $\widetilde\K=\{\widetilde{K}:K\in\K\}$ as in \cite[Section 4.8]{Brenner_Scott} and then be transformed to $\K$ via the maps $\widetilde{\Psi}_K$. 
Namely, by denoting $\widetilde\Omega_h=\bigcup_{K\in\K}\widetilde K$ and $\widetilde\Psi_h:\widetilde\Omega_h\rightarrow \Omega_h$, we can define
$$
(\mathcal{I}^hv) \circ \widetilde\Psi_h := \widetilde{\mathcal{I}}^h(v\circ \widetilde\Psi_h) 
\quad\forall\, v\in L^1(\Omega_h) ,
$$
where $\widetilde{\mathcal{I}}^h$ denotes the Scott--Zhang interpolation on the flat triangulation $\widetilde\K$. 
Since the maps $\widetilde{\Psi}_h$ induces norm equivalence on every simplex, as a result of \eqref{map-Psi_K}, 
we have the following result. 

\begin{lemma}[Scott--Zhang interpolation]\label{Scott-Zhang}
There is a global interpolation operator
$$\mathcal{I}^h:L^1(\Omega_h)\to S_h(\Omega_h)$$
such that 
$$
|u-\mathcal{I}^h u|_{W^{i,p}_h(\Omega_h)}\leq Ch^{k-i}\|u\|_{W^{k,p}_h(\Omega_h)} 
\quad \forall \, 0\leq i\leq k, \,\,\forall\, 1\leq k\leq r+1, \,\,\forall\, u\in W^{k,p}_h(\Omega_h) .
$$
\end{lemma}

The inverse estimate for isoparametric finite elements follows from Lemma \ref{Lemma:basic}, Part 1.  This is presented in the following lemma. 

\begin{lemma}[Inverse estimate]\label{a00}
For $1\leq k\leq l\leq r+1$ and $1\leq p,q\leq \infty$ the following estimate holds: 
\begin{align}\label{eq-inverse}
    \|\uh\|_{W^{l,p}(\check{K})}\leq C h^{k-l+N/p-N/q}\|\uh\|_{W^{k,q}(\check{K})}
    \quad\forall\, \uh\in  S_h(\Omega) ,\,\,\forall\, \check{K}\in \check{\K} . 
\end{align}
\end{lemma}

The following lemma says that the $(r+1)$th-order derivative of a finite element function in $S_h(\Omega)$ can be bounded by its lower-order derivatives. This result is often used to prove a super-approximation property which is stated in Lemma \ref{a2} for iso-parametric finite elements. 

\begin{lemma}\label{a'3}
The following result holds for iso-parametric finite element functions in $S_h(\Omega)${\rm:} 
\begin{align}
    |D^{r+1}\check{v}_h|(x)\leq C\sum_{i=1}^r|D^i\check{v}_h|(x)\quad \forall\, x\in \check{K}
   ,\,\, \forall\,\check{K}\in\check{\K},\,\,\forall\, \check{v}_h\in S_h(\Omega) . 
\end{align}
\end{lemma}
\begin{proof}
Let $M_K:= \Phi_{h,K}\circ \widetilde{\Psi}_K$, which is a diffeomorpshism between the flat simplex $\widetilde K$ and the curved simplex $\check{K}$ (according to Lemma \ref{Lemma:basic}), satisfying the following estimates:  
$$
\|D^s M_K\|_{L^\infty(\widetilde{K})}\leq C 
\quad\mbox{and}\quad \|D^s M^{-1}_K\|_{L^\infty(\check{K})}\leq C\quad \forall\, 1\leq s\leq r+1 . 
$$
According to the definition of $S_h(\Omega)$, a function $\check{v}_h$ is in $S_h(\Omega)$ if and only if the pull-back function $\check{v}_h\circ M_K$ is a polynomial degree\,$\le r$ on the flat simplex $\widetilde{K}$. Therefore, from the estimate on higher order derivatives of composed functions (see \cite[Lemma 3]{CiarletRaviart_1973}), we have 
\begin{align*}
    |D^{r+1}\check{v}_h|(x)
    =&|D^{r+1}((\vh\circ M_K)\circ M^{-1}_K)|(x)\\
    \leq&C\sum_{l=1}^{r+1}|D^{l}(\vh\circ M_K)(M_K^{-1}(x))|\sum_{i\in I(l,r+1)}|DM_K^{-1}(x)|^{i_1}|D^2M_K^{-1}(x)|^{i_2}...|D^{r+1}M_K^{-1}(x)|^{i_{r+1}}\\
    \leq& C\sum_{l=1}^{r+1}|D^l(\vh\circ M_K)|(M^{-1}_K(x))\\
    =& C\sum_{l=1}^r |D^{l}(\vh\circ M_K)|(M_K^{-1}(x)) ,
\end{align*}
where 
\begin{align*}
	I(l,r+1):=\{i=(i_1,i_2,...,i_{r+1})\in \mathbb{Z}^{r+1}: i_k\geq 0,\sum_{k=1}^{r+1} i_k=l;\sum_{k=1}^{r+1} ki_k=r+1\}.
\end{align*}
We can estimate $|D^l(\vh\circ M_K)|(M_K^{-1}(x))$ using the same estimate on higher order derivatives of composed function 
\begin{align*}
	&|D^l(\vh\circ M_K)|(M_K^{-1}(x))\\
	\leq& C\sum_{k=1}^l|D^k\vh|(x)\sum_{i\in I(k,l)}|DM_K(M_K^{-1}(x))|^{i_1}|D^2M_K(M_K^{-1}(x))|^{i_2}...|D^{l}M_K(M_K^{-1}(x))|^{i_l}\\
	\leq &C\sum_{k=1}^l|D^k\vh|(x)
\end{align*}
The result of Lemma \ref{a'3} is obtained by combining the two estimates above. 
\end{proof}
The result above is the key to the superapproximation results for the isoparametric case. For the standard elements the $r + 1$ derivative just vanishes.

\begin{lemma}[Super-approximation]\label{a2}
Let $\omega\in C^\infty_0(\mathbb{R}^N)$ be a smooth cut-off function such that $0\leq \omega\leq 1$ and $\rm{supp}(\omega) \cap \Omega\subset\Omega_0\subset\Omega$, with 
$\Omega_0(d):=\{x\in \Omega: {\rm dist}(x,\Omega_0)\leq d\}\subset\Omega_1$ for some $d>h$. Then the following estimate holds for $\check{v}_h\in \mathring S_h(\Omega)${\rm:}
\begin{align*}
&\|\omega \check{v}_h - \check{I}_h(\omega \check{v}_h)\|_{H^1(\Omega_1)}\leq Ch \Big(\sum_{j=1}^r h^{j-1}\|\omega\|_{W^{j,\infty}(\mathbb{R}^N)}\Big) \|\check{v}_h\|_{H^1(\Omega_1)}+Ch^r\|\omega\|_{r+1,\infty}\|\check{v}_h\|_{L^2(\Omega_1)} , \\ 
&
 \|\omega \check{v}_h-\check{I}_h(\omega \check{v}_h)\|_{H^1(\Omega_1)}\leq C\Big(\sum_{j=1}^{r+1}h^{j-1}\|\omega\|_{W^{j,\infty}(\mathbb{R}^N)}\Big)  \|\check{v}_h\|_{L^2(\Omega_1)}   . 
\end{align*}
\end{lemma}
\begin{proof}
Since ${\rm supp}(\omega \check{v}_h)\subset \Omega_0$, it follows that $\check{I}_h(\omega \check{v}_h)$ vanishes on all $\check{K}$ such that $\check{K}\cap \Omega_0=\emptyset$. 
Since  $\Omega_0(d)\subset\Omega_1$, all the simplices $\check{K}$ such that $\check{K}\cap \Omega_0\neq \emptyset$ are contained in $\Omega_1$. Therefore, we have
\begin{align}\label{a2eq0}
\begin{aligned}
\|\omega \check{v}_h -\check{I}_h (\omega \check{v}_h)\|_{H^1(\Omega_1)}^2
&= \sum_{\check{K}\cap \Omega_0\neq\emptyset} \|\omega \check{v}_h -\check{I}_h (\omega \check{v}_h)\|^2_{H^1(\check{K})} \\
&\le \sum_{\check{K}\cap \Omega_0\neq\emptyset} Ch^{2r}\|\omega \check{v}_h \|_{H^{r+1}(\check{K})}^2 \\
&\le \sum_{\check{K}\cap \Omega_0\neq\emptyset} 
    Ch^{2r}\Big( | \check{v}_h |_{H^{r+1}(\check{K})}^2
    +\sum_{i=0}^r\|\omega\|_{W^{r+1-i,\infty}(\R^N)}^2 \|\check{v}_h\|_{H^i(\check{K})}^2\Big) . 
\end{aligned}
\end{align}
The term $| \check{v}_h |_{H^{r+1}(\check{K})}^2$ can be estimated by using Lemma \ref{a'3}, i.e., \begin{align}\label{a2eq1}
   | \check{v}_h |_{H^{r+1}(\check{K})}^2\leq C\sum_{i=1}^{r}| \check{v}_h |_{H^{i}(\check{K})}^2 . 
\end{align}
For $0\le i\le r$, the term $| \check{v}_h |_{H^{i}(\check{K})}$ can be estimated by using the inverse estimate for isoparametric finite element functions (see Lemma \ref{a00}). This yields the first result of Lemma \ref{a2}. 
The second result can be proved similarly. 
\end{proof}

\subsection{The perturbed bilinear form associated to the isoparametric FEM}
\label{sec:B-Ah}


By using the notation $\check{u}_h\circ\Phi_h=u_h$ and $\check{v}_h\circ\Phi_h=v_h$ for $u_h,v_h\in S_h(\Omega_h)$, the following identity holds: 
\begin{align}\label{Ah-uh-vh=0}
\int_{\Omega_h}\nabla u_h\cdot\nabla v_h \d x 
&=
\int_{\Omega_h}\nabla (\check{u}_h\circ\Phi_h)\cdot\nabla (\check{u}_h\circ\Phi_h) \d x \notag\\
&=
\int_{\Omega} A_h\nabla \check{u}_h \cdot \nabla \check{v}_h \d x 
\quad\forall\, \check{v}_h\in \mathring S_h(\Omega) ,
\end{align}
where
$$
A_h= (\nabla\Phi_{h} (\nabla\Phi_{h})^\top J^{-1})\circ \Phi^{-1}_h  
$$
is a piecewise smooth (globally discontinuous) and symmetric matrix-valued function, and $J=\det(\nabla\Phi_h) \in L^\infty(\Omega_h)$ is the Jacobian of the mapping $\Phi_h:\Omega_h\rightarrow\Omega$, piecewisely defined on every simplex $K\in\K$. 
Therefore, a function $u_h\in S_h(\Omega_h)$ is discrete harmonic if and only if 
\begin{align}\label{Ah-uh-vh=0}
\int_{\Omega} A_h\nabla \check{u}_h \cdot \nabla \check{v}_h \d x = 0 
\quad\forall\, \check{v}_h\in \mathring S_h(\Omega) ,
\end{align}
Identity \eqref{Ah-uh-vh=0} will be used frequently in the following proof.

Since the map $\Phi_h:\Omega_h\rightarrow\Omega$ is close to the identity map ${\rm Id}:\R^N\rightarrow\R^N$ (which satisfies ${\rm Id}(x)\equiv x$), it follows that the matrix $A_h$ is close to the identity matrix. In particular, the following results are corollaries of the second statement of Lemma \ref{Lemma:basic}: 
\begin{align}\label{eq:perturb}
 \|\nabla^j(\Phi_h-{\rm Id})\|_{L^\infty(\Omega_h)}\leq Ch^{r+1-j} \quad\mbox{and}\quad 
 \|A_h-I\|_{L^\infty(\Omega)}\leq Ch^r ,\quad\mbox{for}\,\,\, j=0,1 .
\end{align}
Therefore, for sufficiently small mesh size $h$, the perturbed bilinear form $\check{B}_h: H^1(\Omega)\times H^1(\Omega)\rightarrow\R$ defined by 
\begin{align}\label{perturbed-bilinear}
\check{B}_h(v,\chi)=\int_{\Omega}A_h\nabla v\cdot\nabla \chi\d x 
\end{align}
is continuous and coercive on $H^1_0(\Omega)$, i.e.,
\begin{align}\label{B-coercivity}
\begin{aligned}
&\check{B}_h(v,\chi) \le C\|\nabla v\|_{L^2(\Omega)}\|\nabla \chi\|_{L^2(\Omega)}
&&\forall\, v,\chi\in H^1(\Omega) , \\
&\check{B}_h(v,v) \ge C^{-1} \|\nabla v\|_{L^2(\Omega)}^2 \sim \|v\|_{H^1(\Omega)}^2
&&\forall\, v\in H^1_0(\Omega). 
\end{aligned}
\end{align}

More precisely, the difference between $\check{B}_h(u,v)$ and $B(u,v)$ is estimated in the following lemma. 
\begin{lemma}\label{variation-crime}
There exists a positive constant $h_1>0$ such that for $h\le h_1$ the following result holds: 
If $1\leq p,q\leq \infty$, $\frac{1}{p}+\frac{1}{q}=1$, and $u\in W^{1,p}(\Omega)$, $v\in W^{1,q}(\Omega)$, then 
$$\left|\check{B}_h(u,v)-B(u,v)\right|\leq Ch^r\|\nabla u\|_{L^p(\Lambda_h)}\|\nabla v\|_{L^q(\Lambda_h)}$$
where $ \Lambda_h:=\{x\in \Omega:{\rm dist}(x,\partial\Omega)\leq 2h\}$. 
\end{lemma}
\begin{proof}
Since $\Phi_h={\rm Id}$ at all interior simplices, it follows that $A_h\circ\Phi_h=I$ outside the subdomain $D_h=\{x\in \Omega_h:{\rm dist}(x,\partial\Omega_h)\leq h\}$. Correspondingly, 
$A_h=I$ outside the subdomain $\Phi_h(D_h)$ and therefore,
\begin{align*}
\left|\check{B}_h(u,v)-B(u,v)\right|
&\leq \|A_h-I\|_{L^\infty(\Phi_h(D_h)} \|\nabla u\|_{L^p(\Phi_h(D_h))}\|\nabla v\|_{L^q(\Phi_h(D_h))}\\
&\leq Ch^r\|\nabla u\|_{L^p(\Phi_h(D_h))}\|\nabla v\|_{L^q(\Phi_h(D_h))}.
\end{align*}
If $x\in D_h$, then there exists $x'\in\partial\Omega_h$ such that $|x-x'|={\rm dist}(x,\partial\Omega_h)\le h$ and  
$$
|\Phi_h(x)-\Phi_h(x')|
\le 
|\Phi_h(x) -x |+ |x-x'| + |x'-\Phi_h(x')|
\le Ch^{r+1} + h + Ch^{r+1} ,
$$
which implies that 
$$
{\rm dist}(\Phi_h(x),\partial\Omega) 
\le Ch^{r+1} + h .
$$
For sufficiently small $h$ we obtain ${\rm dist}(\Phi_h(x),\partial\Omega) \le 2h$ and therefore $\Phi_h(D_h)\subset\Lambda_h$. 
\end{proof}

\subsection{Reduction of the problem} \label{sec: start of the proof}

Let $x_0\in\overline\Omega$ be a point satisfying 
$$
|\check{u}_h(x_0)|=\|\check{u}_h\|_{L^\infty(\Omega)}
\quad\mbox{with}\quad d={\rm dist}(x_0,\partial\Omega) .
$$

If $d\ge 2kh$ for some fixed $k\ge 1$, i.e., $x_0$ is relatively far away from the boundary $\partial\Omega$, then we can choose $\Omega_1=\{x_0\}$ and $\Omega_2=S_{d/2}(x_0)$ and use the interior $L^\infty$ estimate established in \cite[Corollary 5.1]{Schatz-Wahlbin-1977}. This yields the following result: 
$$
| \uh(x_{0})|
\le
C d^{-\frac{N}2} \| \uh\|_{L^{2}(S_{d}(x_{0}))} . 
$$

Otherwise, we have $d < 2kh$. In this case, assuming that $x_0\in\check{K}$ for some curved simplex $\check{K}\in\check{\K}$, by the inverse estimate in Lemma \ref{a00} we have
$$
| \uh(x_{0})|
=
\|\uh\|_{L^\infty(\check{K})}
\le
C h^{-\frac{N}2}\|\uh\|_{L^2(\check{K})}
\leq 
C h^{-\frac{N}2}
\| \uh\|_{L^{2}(S_{2kh}(x_{0}))}
.
$$

Overall, for either $d\ge 2kh$ or $d<2kh$, the following estimate holds:
\begin{align}\label{Linfty-1}
| \uh(x_{0})|
\le
C \rho^{-\frac{N}2}
\| \uh\|_{L^{2}(S_{\rho}(x_{0}))},\quad\mbox{with}\quad \rho=d+2kh.
\end{align}


To estimate the term $\| \uh\|_{L^{2}(S_{\rho}(x_{0}))}$ on the right-hand side of \eqref{Linfty-1}, we use the following duality property:
$$
\| \uh\|_{L^{2}(S_{\rho}(x_{0}))}
= \sup_{
\begin{subarray}{c}
{\rm supp}(\varphi)\subset S_{\rho}(x_{0}) \\
\|\varphi\|_{L^{2}(S_{\rho}(x_{0}))}\le 1
\end{subarray}
}|( \uh,\varphi)| ,
$$
where $(\cdot,\cdot)$ denotes the inner product of $L^2(\Omega)$ 
(or $L^2(\Omega)^N$ for vector-valued functions), i.e., 
$$
(u,v):=\int_{\Omega}u\cdot v \d x . 
$$
Hence, there exists a function $\varphi\in C_0^\infty(\Omega)$ with the following properties:
\begin{align}\label{L2-varphi}
&{\rm supp}(\varphi)\subset S_{\rho}(x_{0}),
\quad
\|\varphi\|_{L^{2}(S_{\rho}(x_{0}))}\le 1 ,\\
&\label{L2-Srho-uh}
\| \uh\|_{L^{2}(S_{\rho}(x_{0}))}
\le 2|( \uh,\varphi)| .
\end{align}
For this function $\varphi$, we define $v\in H^1_0(\Omega)$ and $u\in H^1(\Omega)$ to be the solutions of the following elliptic equations (in the weak form): 
\begin{align}\label{weak-PDE-v}
\left\{\begin{aligned}
(A_h \nabla v,\nabla\chi) &=(\varphi,\chi) 
&&\forall\,\chi\in H^1_0(\Omega),\\
v&= 0 &&\mbox{on}\,\,\,\partial\Omega ,
\end{aligned}\right.
\end{align}
and 
\begin{align}\label{weak-form-u}
\hspace{-15pt}
\left\{\begin{aligned}
(A_h\nabla u,\nabla \chi)&=0 &&\forall\, \chi\in H^1_0(\Omega),\\
u&= \uh &&\mbox{on}\,\,\,\partial\Omega , 
\end{aligned}\right.
\end{align}
respectively. 
The maximum principle of the continuous problem \eqref{weak-form-u} implies that 
\begin{align}\label{max-principle-PDE}
\|u\|_{L^{\infty}(\Omega)} \leqslant \| \uh\|_{L^{\infty}(\partial \Omega)}.
\end{align}
Therefore, we have
\begin{align}\label{L2-uh-1}
\| \uh\|_{L^{2}(S_{\rho}(x_{0}))}
&\le
2|( \uh,\varphi)|
&&\mbox{(here we have used \eqref{L2-Srho-uh})} \nonumber \\
&=2|( \uh-u,\varphi) + (u,\varphi)| \nonumber \\
&= 2|(A_h\nabla( \uh-u),\nabla v) + (u,\varphi)|
&&\mbox{(here we have used \eqref{weak-PDE-v})} \nonumber\\
&=2| (A_h\nabla  \uh ,\nabla v) + (u,\varphi)|
&&\mbox{(here we have used \eqref{weak-form-u})} \nonumber\\
&\le 2|(A_h\nabla  \uh ,\nabla v)| + 2\|u\|_{L^\infty(\Omega)}\|\varphi\|_{L^1(S_\rho(x_0))} 
&&\mbox{(since ${\rm supp}(\varphi)\subset S_{\rho}(x_{0})$)}\nonumber\\
&\le 2|(A_h\nabla  \uh ,\nabla v)| + C\rho^{\frac{N}2} \| \uh\|_{L^\infty(\partial\Omega)}\|\varphi\|_{L^2(S_{\rho}(x_{0}))} , \end{align}
where we have used \eqref{max-principle-PDE} and the H\"older inequality in deriving the last inequality. Combing inequalities \eqref{Linfty-1} and \eqref{L2-uh-1}, we have
\begin{align}\label{linfty0}
    \|\uh\|_{L^\infty(\Omega)}=|\check{u}_h(x_0)|\leq C\rho^{-\frac{N}2}|(A_h\nabla\uh,\nabla v)|+C\|\uh\|_{L^\infty(\partial\Omega)}
\end{align}
where we have used the fact that $\|\varphi\|_{L^2(S_\rho(x_0))}\leq 1$. 

It remains to estimate $\rho^{-\frac{N}{2}}|(A_h\nabla \uh,\nabla v)|$. 
To this end, we define $R_h: H^1_0(\Omega)\rightarrow \sh(\Omega)$ to be the Ritz projection associated with the perturbed bilinear form defined in \eqref{perturbed-bilinear}, i.e., 
\begin{align}\label{def-Ritz-proj}
\left(A_h\nabla (v-R_h v),\nabla \check{\chi}_h\right) = 0 
\quad \forall\, \check{\chi}_h\in \mathring S_h(\Omega) ,
\end{align}
which is well defined in view of the coercivity of the bilinear form; see \eqref{B-coercivity}. 
By using identity \eqref{Ah-uh-vh=0} for the discrete harmonic function $u_h$ and the definition of the Ritz projection $R_h$ in \eqref{def-Ritz-proj}, we have 
\begin{align}\label{Ah-graduh-gradv}
    (A_h\nabla \uh,\nabla v)&=(A_h\nabla \uh, \nabla(v-R_hv))\nonumber\\
    &= (A_h\nabla (\uh-\xh), \nabla(v-R_hv))\quad\forall\xh\in\sh(\Omega) . 
\end{align}
In particular, we can choose $\check{\chi}_h=\chi_h\circ\Phi_h^{-1} \in \mathring S_h(\Omega)$ to satisfy $\chi_h=u_h$ on all interior Lagrange nodes while $\chi_h=0$ on all the boundary nodes (which implies $\chi_h=0$ on $\partial\Omega_h$ and therefore $\check{\chi}_h\equiv 0$ on $\partial\Omega$). Then 
\begin{align}\label{xh-lemma}
    \|\xh-\uh\|_{L^\infty(\Omega)}\leq C\|\uh\|_{L^\infty(\partial\Omega)}  .
\end{align}
Let $\Lambda_h=\{x\in \Omega: {\rm dist}(x,\partial\Omega)\le 2h\}$ be a neighborhood of the boundary $\partial\Omega$, when $h$ sufficiently small, $\uh-\xh=0$ outside $\Lambda_h$. Then 
\begin{align}\label{A_h-grad-uh-chih}
    |(A_h\nabla (\uh-\xh), \nabla(v-R_hv))| &\leq C\|\nabla(\xh-\uh)\|_{L^\infty(\Omega)}\|\nabla(v-R_hv)\|_{L^1(\Lambda_h)}\nonumber\\
    &\leq Ch^{-1}\|\uh\|_{L^\infty(\partial\Omega)}\|\nabla(v-R_hv)\|_{L^1(\Lambda_h)} , 
\end{align}
where we have used \eqref{xh-lemma} and the inverse estimate for finite element functions. Substituting \eqref{Ah-graduh-gradv} and \eqref{A_h-grad-uh-chih} into \eqref{linfty0}, we obtain 
\begin{align}\label{Linfty-uh-2}
\| \uh\|_{L^\infty(\Omega)}
\le
C\big(\rho^{-\frac{N}2} h^{-1} \|\nabla (v- R_hv)\|_{L^1(\Lambda_h)}
+1)
\| u_h\|_{L^\infty(\partial\Omega)} .
\end{align}
The proof of Theorem \ref{THM1} will be completed if the following result holds:  
\begin{align}\label{to-prove}
\rho^{-\frac{N}2} h^{-1} \|\nabla (v- R_hv)\|_{L^1(\Lambda_h)}\le C,
\end{align}
which will be proved in the following subsections.

\subsection{Regularity decomposition}
\label{section:Reguliarty}

In order to estimate the left-hand side of \eqref{to-prove}, we need to use a local energy estimate and a duality argument, which is based on the regularity result of the following elliptic equation (in the weak form): Find $v\in H^1_0(\Omega)$ such that 
\begin{align}\label{weak-PDE}
(A_h\nabla v,\nabla\chi)=(f,\chi)
\quad\forall\,\chi\in H^1_0(\Omega) ,
\end{align}
where $A_h$ is a globally discontinuous matrix-valued function defined in Section \ref{sec:B-Ah}.

Due to the discontinuity of the coefficient matrix $A_h$, the standard $H^2$ regularity does not hold for the elliptic equation \eqref{weak-PDE}. We decompose the solution $v\in H^1_0(\Omega)$ of equation \eqref{weak-PDE} into the following two parts: 
\begin{align}\label{v-decomp}
    v=v_1+v_2 , 
\end{align}
where $v_1\in H^1_0(\Omega)$ and $v_2\in H^1_0(\Omega)$ are the weak solutions of the equations 
\begin{align}\label{weak-pde-v1}
&    (\nabla v_1, \nabla\chi)=(f,\chi)&& \forall \chi\in H^1_0(\Omega) ,\\
\label{weak-pde-v2}
&    (A_h\nabla v_2,\nabla\chi)= ((I-A_h)\nabla v_1, \nabla \chi)&&\forall\chi\in H^1_0(\Omega) .
\end{align}
Equation \eqref{weak-pde-v1} has a constant coefficient and therefore the classical $W^{2,q}$ regularity estimate holds for $1<q<2+\varepsilon$, for some $\varepsilon>0$ which depends on the interior angles at the edges and corners of the domain $\Omega$ (see \cite[Corollaries 3.7, 3.9 and 3.12]{1992-Dauge}), i.e., 
\begin{align}\label{eq:r1}
    \|v_1\|_{W^{2,q}(\Omega)}\leq C_q \|f\|_{L^q(\Omega)} \quad\forall\, 1<q<2+\varepsilon. 
\end{align}
Since equation \eqref{weak-pde-v2} has discontinuous coefficients, the $W^{2,q}$ regularity estimate does not hold. We have to estimate $v_2$ by using the $W^{1,p}$ estimate in the following lemma.  

\begin{lemma}\label{gradient-lp}
For every $1<p<\infty$ there exists $h_p>0$ {\rm(}which depends on $p${\rm)}, such that for $h\leq h_p$, the solution $w\in H^1_0(\Omega)$ of the equation 
\begin{align}\label{weak-pde-v}
  (A_h\nabla w, \nabla\chi)=(\vec{g},\nabla\chi) \quad\forall \chi\in H^1_0(\Omega)  
  \quad\mbox{with}\quad \vec{g}\in L^p(\Omega)^N\cap L^2(\Omega)^N,
\end{align}
satisfies $w\in W^{1,p}(\Omega)$ and 
\begin{align}\label{eq:grad-lp}
    \|w\|_{W^{1,p}(\Omega)}\leq C_p\|\vec{g}\|_{L^p(\Omega)},
\end{align}
where $C_p$ is a constant which is independent of $h$ {\rm(}possibly depending on $p${\rm)}.
\end{lemma}
\begin{proof}
We can rewrite equation \eqref{weak-pde-v} into the following form: 
\begin{align*}
  (\nabla w, \nabla\chi)=(\vec{g},\nabla\chi) + ((I-A_h)\nabla w, \nabla\chi) \quad\forall \chi\in H^1_0(\Omega)  , 
\end{align*}
and apply the $W^{1,p}$ regularity estimate for the Poisson equation (which holds in a smooth domain or curvilinear polyhedron with edge openings smaller than $\pi$; see \cite[Corollaries 3.7, 3.9 and 3.12]{1992-Dauge}). 
This yields the following inequality:  
$$
\|w\|_{W^{1,p}(\Omega)}\leq C_p\|\vec{g}\|_{L^p(\Omega)} + C_p\| I-A_h \|_{L^\infty(\Omega)}\|w\|_{W^{1,p}(\Omega)} . 
$$
Since $\|A_h-I\|_{L^\infty}\leq Ch$, for sufficiently small $h$ (depending on $p$) the last term on the right-hand side can be absorbed by the left-hand side. This yields the result of Lemma \ref{gradient-lp}. 
\end{proof}

By combining the $W^{2,q}$ regularity estimate in \eqref{eq:r1} and the $W^{1,p}$ regularity estimate in Lemma \ref{gradient-lp}, we can prove the following result. 

\begin{lemma}\label{general-W1p}
Let $1<p,q<\infty$ be numbers such that $1/q\leq 1/n+1/p$, and assume that $h\leq h_p$, where $h_p$ is given in Lemma \ref{gradient-lp}. 
Let $w\in H^1_0(\Omega)$ be the weak solution of the equation
\begin{align}\label{weak-pde-general}
    (A_h\nabla w,\nabla\chi)=(f,\chi)+(\vec{g},\nabla\chi)\quad\forall \chi\in H^1_0(\Omega)
\end{align}
for some $f\in L^q(\Omega)\cap L^2(\Omega)$ and $\vec{g}\in L^p(\Omega)^N\cap L^2(\Omega)^N$. 
Then $w\in W^{1,p}(\Omega)$ and 
\begin{align}\label{eq:grad-lpq}
    \|w\|_{W^{1,p}(\Omega)}\leq C_p\|f\|_{L^q(\Omega)}+C_p\|\vec{g}\|_{L^p(\Omega)} . 
\end{align}
\end{lemma}
\begin{proof}
We consider the decomposition $w=w_1+w_2$ with $w_1,w_2\in H^1_0(\Omega)$ weakly solving
\begin{align*}
\begin{aligned}
& (\nabla w_1, \nabla\chi)=(f,\chi)&& \forall \chi\in H^1_0(\Omega) , \\ 
&(A_h\nabla w_2,\chi)= ((I-A_h)\nabla w_1+\vec{g}, \nabla \chi)&&\forall\chi\in H^1_0(\Omega) .
\end{aligned}
\end{align*}
Note that for $\chi\in W^{1,p'}_0(\Omega)$ where $1/p+1/p'=1$
\begin{align*}
	|(f,\chi)|\leq& \|f\|_{L^q(\Omega)}\|\chi\|_{L^{q'}(\Omega)}\quad\mbox{($1/q+1/q'=1$)}\\
	\leq& C\|f\|_{L^q(\Omega)}\|\chi\|_{W^{1,p'}(\Omega)}\quad\mbox{(embedding $W^{1,p'}\hookrightarrow L^{q'}$ used)},
\end{align*}
therefore we have $\|f\|_{W^{-1,p}(\Omega)}\leq C\|f\|_{L^q(\Omega)}$. By the $W^{1,p}$ regularity estimate for the Poisson equation on curvilinear polyhedron (see \cite[Corollaries 3.7, 3.9 and 3.12]{1992-Dauge}), there holds
\begin{align*}
	\|w_1\|_{W^{1,p}(\Omega)}\leq C_p\|f\|_{W^{-1,p}(\Omega)}\leq C_p\|f\|_{L^q(\Omega)}
\end{align*}
Then we apply the $W^{1,p}$ estimate in Lemma \ref{gradient-lp} to the equation of $w_2$. This yields the following result: 
$$
\|w_2\|_{W^{1,p}(\Omega)}
\leq C_p\|\vec{g}+(I-A_h)\nabla w_1\|_{L^p(\Omega)}
\leq C_p\|\vec{g}\|_{L^p(\Omega)}+C_p\|f\|_{L^q(\Omega)} . 
$$
The result of Lemma \ref{general-W1p} follows from combining the estimates for $w_1$ and $w_2$.
\end{proof}

The following lemma was proved in \cite[Lemma 2.2]{Leykekhman_Li_2021} for polyhedral domains. The proof of this result for smooth domains and curvilinear polyhedron is the same. 

\begin{lemma}\label{W1p0d}
If $\chi\in W^{1,p}_0(\Omega)$ for some $1<p<\infty$ and $x^*\in \partial\Omega$, then
$$
\|\chi\|_{L^{p}(S_{d_*}(x^*))}
\le
Cd_*\|\nabla\chi\|_{L^{p}(\Omega)} , 
$$
where $S_{d_*}(x^*):=\{x\in\Omega:|x-x^*|<d_*\}$.
\end{lemma}

\begin{lemma}\label{lemma:w1p-d}
Let $1<p<\infty$ and $h\leq h_p$, where $h_p$ is given in Lemma \ref{gradient-lp}. 
For 
$$f\in L^p(\Omega)\cap L^2(\Omega)\,\,\,\mbox{with}\,\,\, {\rm supp}(f)\subset S_{d_*}(x_0),
\,\,\,\mbox{where\, $x_0\in \overline{\Omega}$\, and\, ${\rm dist}(x_0,\partial\Omega)\leq d_*$},$$ 
the solution $v\in H^1_0(\Omega)$ of equation \eqref{weak-PDE} satisfies
\begin{align}
    \|v\|_{W^{1,p}(\Omega)}\leq C_p d_*\|f\|_{L^p(\Omega)}
\end{align}
\end{lemma}
\begin{proof}
We consider the decomposition $v=v_1+v_2$ in \eqref{v-decomp}--\eqref{weak-pde-v2}. If ${\rm dist}(x_0,\partial\Omega)\leq d_*$, then $S_{d_*}(x_0)\subset S_{2d_*}(\bar{x}_0)$ for some $\bar{x}_0\in \partial\Omega$. Note that for $\chi\in W^{1,p'}_0(\Omega)$ where $1/p+1/p'=1$, we have
\begin{align*}
|(f,\chi)|\leq& \|f\|_{L^p(S_{d_*}(x_0))}\|\chi\|_{L^{p'}(S_{d_*}(x_0))}\\
	\leq& \|f\|_{L^p(S_{d_*}(x_0))}\|\chi\|_{L^{p'}(S_{2d_*}(\bar{x}_0))}\\
	\leq& Cd_*\|f\|_{L^p(\Omega)}\|\nabla \chi\|_{L^{p'}(\Omega)},\quad \mbox{(Lemma \ref{W1p0d} used)}
\end{align*}
which implies that $\|f\|_{W^{-1,p}(\Omega)}\leq Cd_*\|f\|_{L^p(\Omega)}$. Thus by the $W^{1,p}$ regularity estimate for the Poisson equation on curvilinear polyhedron (see \cite[Corollaries 3.7, 3.9 and 3.12]{1992-Dauge}), there holds:
$$
\|v_1\|_{W^{1,p}(\Omega)}\leq C_p\|f\|_{W^{-1,p}(\Omega)}\leq C_p d_*\|f\|_{L^p(\Omega)} . 
$$
By applying Lemma \ref{gradient-lp} to equation \eqref{weak-pde-v2}, we obtain 
\begin{align*} 
    \|v_2\|_{W^{1,p}(\Omega)}&\leq C_p\|(I-A_h)\nabla v_1\|_{L^p(\Omega)}
    \leq C_ph\|v_1\|_{W^{1,p}(\Omega)} 
    \leq C_phd_*\|f\|_{L^p(\Omega)} .
\end{align*}
The last two inequalities imply the result of Lemma \ref{lemma:w1p-d}. 
\end{proof}


The next lemma is about the Cacciopoli inequality for harmonic functions which is the same as in \cite[Lemma 8.3]{Schatz-Wahlbin-1978}. 
The result holds for smooth domains and curvilinear polyhedra on which the elliptic $H^2$ regularity result holds for the Poisson equation. 

\begin{lemma}\label{lemma: Caccipoli}
Let $D$ and $D_d$ be two subdomains of $\Omega$ satisfying $D \subset D_d\subset \Omega$, with
$$
D_d=\{x\in \Omega:
dist(x, D)\le d\} ,
$$
where $d$ is a positive constant.
If $v\in H^1_0(\Omega)$ and $v$ is harmonic on $D_d$, i.e.
$$
(\nabla v, \nabla w)=0 \quad \forall w\in H^1_0(D_d) ,
$$
then the following estimates hold:
\begin{subequations}          \label{eq: Caccipoli}
\begin{align}
|v|_{H^2(D)}&\le Cd^{-1}\|v\|_{H^1(D_d)} \label{eq: Caccipoli 1},\\
\| v\|_{H^1(D)}&\le Cd^{-1}\|v\|_{L_2(D_d)}\label{eq: Caccipoli 2}.
\end{align}
\end{subequations}
\end{lemma}

We also need the following interior estimate in the estimation of $v_2$. 

\begin{lemma}\label{weak-cappo}
Let $1<p,q<\infty$ be numbers such that $1/q\leq 1/n+1/p$ and assume that $h\leq h_p$, where $h_p$ is given in Lemma \ref{gradient-lp}. Let $D\subset D_d\subset \Omega$ be subdomains, with 
$
D_d=\{x\in \Omega:
dist(x, D)\le d\}.
$
If $v\in W^{1,p}_0(\Omega)\cap H^1_0(\Omega)$ satisfies equation
\begin{align}\label{eq-quasi-harmonic}
\hspace{12.5pt} (A_h\nabla v,\nabla \chi)=0 \quad\forall \chi\in H^1_0(D_d) , 
\end{align}
or 
\begin{align}\label{eq-harmonic}
    (\nabla v,\nabla \chi)=0 \quad\forall \chi\in H^1_0(D_d).
\end{align}
Then 
\begin{align}
    \|v\|_{W^{1,p}(D)}\leq \frac{C_p}{d} (\|v\|_{L^p(D_d)}+ \|v\|_{W^{1,q}(D_d)}) . 
\end{align}
\end{lemma}
\begin{proof}
We focus on the first case: $v$ satisfies equation \eqref{eq-quasi-harmonic}. The proof for the second case is the same and therefore omitted. 

First, we choose a cut-off function $\omega\in C^\infty_0(\mathbb{R}^N)$, $\omega\equiv 1$ on $D$, ${\rm supp}(\omega)\cap \Omega\subset D_d$, with $\|\omega\|_{W^{1,\infty}(\R^N)}\leq Cd^{-1}$. Then $\omega v\in H^1_0(\Omega)$ satisfies the following equation: 
\begin{align*}
    (A_h\nabla (\omega v),\nabla \chi)&= (\omega A_h\nabla v,\nabla\chi)+(A_h\nabla\omega, v\nabla\chi)\\
    &=(A_h\nabla v, \nabla (\omega\chi))- (A_h\nabla v, \chi\nabla\omega)+(A_h\nabla\omega,v\nabla\chi)\\
    &=(A_hv\nabla\omega,\nabla \chi)-(A_h\nabla v\cdot \nabla\omega, \chi)\quad \forall\chi\in H^1_0(\Omega)
\end{align*}
where we have used the identity $(A_h\nabla v, \nabla (\omega\chi))=0$ in the derivation of the last equality, which is a consequence of \eqref{eq-quasi-harmonic} and $\omega\chi\in H^1_0(D_d)$. Then we can apply Lemma \ref{general-W1p} to the above equation satisfied by $\omega v$. This yields the following result:
\begin{align*}
 \|\omega v\|_{W^{1,p}(\Omega)}&\leq C_p\|A_h v\nabla\omega\|_{L^p(\Omega)}+C_p\|A_h\nabla v\cdot\nabla \omega\|_{L^q(\Omega)}\\
    &\leq \frac{C_p}{d}\|v\|_{L^p(D_d)}+\frac{C_p}{d}\|v\|_{W^{1,q}(D_d)}.
\end{align*}
Since $\omega=1$ on $D$, the last inequality implies the result of Lemma \ref{weak-cappo}.
\end{proof}

\begin{lemma}\label{v2-weak-cappo}
Let $1<p,q<\infty$ be numbers such that $1/q\leq 1/n+1/p$ and assume that $h\leq \min \{h_p,h_q\}$, where $h_p,h_q$ are given in Lemma \ref{gradient-lp}. Let $D\subset D_d\subset \Omega$ be subdomains, with $D_d=\{x\in \Omega:dist(x, D)\le d\}$. 
If the source function $f$ has ${\rm supp}(f)\cap D_d=\emptyset$, then the solution $v_2$ of equation \eqref{weak-pde-v2} satisfies the following estimate: 
\begin{align}
    \|v_2\|_{W^{1,p}(D)}\leq \frac{C_{p,q}}{d}h\|v_1\|_{W^{1,q}(\Omega)} . 
\end{align}
\end{lemma}

\begin{proof}
We consider a cut-off function $\omega$ such that $\omega\equiv 1$ in $D$ and ${\rm supp}(\omega)\subset D_{d/2}$, with $\|\omega\|_{W^{1,\infty}(\R^N)}\leq Cd^{-1}$. Then the following equation can be written down similarly as in the proof of Lemma \ref{weak-cappo}: 
\begin{align*}
    (A_h\nabla (\omega v_2),\nabla \chi)&= (\omega (I-A_h)\nabla v_1,\nabla\chi)+((I-A_h)\nabla v_1\cdot\nabla\omega, \chi)\\
    & \quad +(v_2 A_h\nabla \omega, \nabla \chi)- (A_h\nabla v_2\cdot \nabla\omega, \chi)\quad \forall\chi\in H^1_0(\Omega).
\end{align*}
By applying Lemma \ref{general-W1p} to the equation above, we obtain
\begin{align}\label{v2-W1p}
    \|v_2\|_{W^{1,p}(D)}
    &\leq 
    C_ph\|v_1\|_{W^{1,p}(D_{d/2})}
    + \frac{C_ph}{d}\|v_1\|_{W^{1,q}(\Omega)}
    +\frac{C_p}{d}\|v_2\|_{L^{p}(\Omega)} +\frac{C_p}{d}\|v_2\|_{W^{1,q}(\Omega)} \notag\\
    &\leq 
    C_ph\|v_1\|_{W^{1,p}(D_{d/2})}
    + \frac{C_ph}{d}\|v_1\|_{W^{1,q}(\Omega)}
    +\frac{C_p}{d}\|v_2\|_{W^{1,q}(\Omega)},
\end{align}
where we have used Sobolev embedding $W^{1,q}(\Omega)\hookrightarrow L^p(\Omega)$. 

Since ${\rm supp}(f)\cap D_d=\emptyset$, it follows that the solution $v_1$ of \eqref{weak-pde-v1} satisfies equation \eqref{eq-harmonic}. Therefore, Lemma \ref{weak-cappo} implies that  
\begin{align*}
    \|v_1\|_{W^{1,p}(D_{d/2})}
    \leq \frac{C_p}{d} (\|v\|_{L^p(D_d)}+ \|v\|_{W^{1,q}(D_d)}) 
    \leq \frac{C_p}{d}\|v_1\|_{W^{1,q}(\Omega)} . 
\end{align*}
By applying Lemma \ref{gradient-lp} to equation \eqref{weak-pde-v2}, we also obtain
\begin{align*}
    \|v_2\|_{W^{1,q}(\Omega)}
    \leq C_{q}\|I-A_h\|_{L^\infty(\Omega)} \|v_1\|_{W^{1,q}(\Omega)} 
    \leq C_{q}h\|v_1\|_{W^{1,q}(\Omega)}. 
\end{align*}
Then, substituting the last two inequalities into \eqref{v2-W1p}, we obtain the result of Lemma \ref{v2-weak-cappo}. 
\end{proof}

\subsection{$W^{1,p}$ stablity of the Ritz projection (with discontinuous coefficients)}
\label{section:W1p}

In \cite{Guzman_Leykekhman_Rossman_Schatz_2009} the $W^{1,\infty}$ stability of the Ritz projection is proved for the Poisson equation in convex polyhedral domains. The proof is based on the following properties of the domain and finite elements:
\begin{enumerate}

\item[(P1)] H\"older estimates of the Green function for the Poisson equation, i.e.,
\begin{align}\label{Holder-Green}
\begin{aligned}
&\frac{\left|\partial_{x_{i}} G(x, \xi)-\partial_{y_{i}} G(y, \xi)\right|}{|x-y|^{\sigma}} 
\leq C\left(|x-\xi|^{-2-\sigma}+|y-\xi|^{-2-\sigma}\right) \\
&\frac{\left|\partial_{x_{i}} \partial_{\xi_{j}} G(x, \xi)-\partial_{y_{i}} \partial_{\xi_{j}} G(y, \xi)\right|}{|x-y|^{\sigma}} 
\leq C\left(|x-\xi|^{-3-\sigma}+|y-\xi|^{-3-\sigma}\right)
\end{aligned}
\end{align}
for $i, j=1,2,3$. 

\item[(P2)] 
Elliptic $H^2$ regularity result for the Poisson equation. 

\item[(P3)] 
Exact triangulation which matches the boundary $\partial\Omega$. 

\item[(P4)] 
Error estimates for the Lagrange interpolation holds as in Lemma \ref{Proposition-A0}. 

\end{enumerate}

Note that the H\"older estimates for the Green function in \eqref{Holder-Green} was proved in \cite{Guzman_Leykekhman_Rossman_Schatz_2009} for general curvilinear polyhedral domains with edge opening smaller than $\pi$, instead of merely classical polyhedral domains. 
If we define a modified Ritz projection $R^*_h$ associated to the Poisson equation (without the discontinuous coefficient $A_h$), i.e., 
\begin{align}\label{Ritz-Poisson}
\int_{\Omega}\nabla (v-R^*_h v)\cdot \nabla\check{\chi}_h \d x 
=0 \quad \forall \check{\chi}_h\in \mathring S_h(\Omega) ,
\end{align}
then all the properties in (P1)--(P4) are possessed by the curvilinear polyhedral domain $\Omega$ and the finite element space $\mathring S_h(\Omega)$. The latter is based on the triangulation $\check{K}$ which matches the boundary $\partial\Omega$ exactly. Therefore, the $W^{1,\infty}$ stability still holds for the modified Ritz projection defined in \eqref{Ritz-Poisson}. The result is stated in the following lemma. 

\begin{lemma}\label{w1,inf-stable}
\begin{align}
    \|R^*_h v\|_{W^{1,\infty}(\Omega)}\leq C\|v\|_{W^{1,\infty}(\Omega)}
    \quad\forall\, v\in H^1_0(\Omega)\cap W^{1,\infty}(\Omega) .
\end{align}
\end{lemma}

By real interpolation between the $H^1$ and $W^{1,\infty}$ stability estimates (see \cite[result in (5.1)]{Milman}), we obtain the $W^{1,p}$ stability of the modified Ritz projection for $2\leq p\leq \infty$. The result can also be extended to $1<p\leq 2$ by a duality argument as in \cite[Section 8.5]{Brenner_Scott}, which requires Poisson equation to have the $W^{1,p'}$ regularity (this is true for a curvilinear polyhedron with edge opening smaller than $\pi$). The result is summarized below. 

\begin{lemma}[$W^{1,p}$ stability of the modified Ritz projection $R_h^*$]
For any $1<p\le\infty$, there exists a positive constant $h_p$ such that for $h\le h_p$ the following result holds:
\label{Lemma:W1p_Rh*}
\begin{align}
    \|R^*_h u\|_{W^{1,p}(\Omega)}\leq C_p\|u\|_{W^{1,p}(\Omega)}\quad\forall u\in W^{1,p}(\Omega)\cap H^1_0(\Omega)  . 
\end{align}
\end{lemma}

By a ``perturbation" argument, similar as \cite[Section 8.6]{Brenner_Scott}, one can obtain the $W^{1,p}$ stability of the Ritz projection $R_h$. This is stated in the following proposition.

\begin{proposition}[$W^{1,p}$ stability of the Ritz projection $R_h$]\label{w1,p-stable}
For any $1<p<\infty$, there exists a positive constant $h_p$ such that for $h\le h_p$ the following result holds:
\begin{align}
    \|R_h u\|_{W^{1,p}(\Omega)}\leq C_p\|u\|_{W^{1,p}(\Omega)}\quad\forall u\in W^{1.p}(\Omega)\cap H^1_0(\Omega).  
\end{align}

\end{proposition}
\begin{proof}
For $v\in H^1_0(\Omega)$, its Ritz projection $R_hv\in \mathring S_h(\Omega) $ satisfies the following equation: 
\begin{align*}
\int_{\Omega}\nabla (v-R_h v)\cdot \nabla\check{\chi}_h \d x 
= \int_{\Omega} (I-A_h)\nabla (v-R_h v)\cdot \nabla\check{\chi}_h \d x  \quad \forall \check{\chi}_h\in \mathring S_h(\Omega) .
\end{align*}
If we define $w$ to be the solution of the following elliptic equation (in the weak form):
\begin{align*}
\int_{\Omega}\nabla w\cdot \nabla\check{\chi} \d x 
= -\int_{\Omega} (I-A_h)\nabla (v-R_h v)\cdot \nabla\check{\chi} \d x  \quad \forall \check{\chi}\in H^1_0(\Omega) ,
\end{align*}
then 
\begin{align*}
\int_{\Omega}\nabla (w+v-R_h v)\cdot \nabla\check{\chi}_h \d x 
= 0  \quad \forall \check{\chi}_h\in \mathring S_h(\Omega) ,
\end{align*}
which means that $R_hv=R_h^*(w+v)$. Lemma \ref{Lemma:W1p_Rh*} implies that 
\begin{align*}
    \|R_hv\|_{W^{1,p}(\Omega)}
    =\|R_h^*(w+v)\|_{W^{1,p}(\Omega)}
    &\le C_p\|w+v\|_{W^{1,p}(\Omega)}\\
    &\le C_p\|I-A_h\|_{L^\infty(\Omega)}\|v-R_h v\|_{W^{1,p}(\Omega)}+C_p\|v\|_{W^{1,p}(\Omega)}\\
    &\le C_ph\|R_h v\|_{W^{1,p}(\Omega)}+C_p\|v\|_{W^{1,p}(\Omega)} .
\end{align*}
There exists a constant $h_p$ such that for $h\le h_p$ the first term on the right-hand side can be absorbed by the left-hand side. In this case we obtain the result of Proposition \ref{w1,p-stable}.
\end{proof}

As a result of Proposition \ref{w1,p-stable}, we obtain the following $W^{1,p}$ error estimate for the Ritz projection.

\begin{lemma}\label{Lemma:Ritz-error-W1p}
For any $1<q<2+\varepsilon$, there exists a positive constant $h_q$ such that for $h\le h_q$ the solution of  equation \eqref{weak-PDE} has the following error bound{\rm:} 
\begin{align*}
\|v- R_hv\|_{W^{1,q}(\Omega)}
\le
C_q h \|f\|_{L^q(\Omega)}
\quad\forall\, f\in L^q(\Omega)\cap L^2(\Omega). 
\end{align*}
\end{lemma}

\begin{proof}
We consider the decomposition $v=v_1+v_2$ in \eqref{v-decomp}--\eqref{weak-pde-v2}. The $W^{2,q}$ estimate in \eqref{eq:r1} and the $W^{1,p}$ estimate in Lemma \ref{gradient-lp} imply that $v_1$ and $v_2$ satisfy the following estimates: 
\begin{align*}
    &\|v_1\|_{W^{2,q}(\Omega)}\leq C_q\|f\|_{L^q(\Omega)} \quad\forall\, 1<q<2+\varepsilon , \\
    &\|v_2\|_{W^{1,q}(\Omega)}
    \leq C_qh \|v_1\|_{W^{1,q}(\Omega)} 
    \leq C_qh\|f\|_{L^q(\Omega)}.
\end{align*}
Applying the $W^{1,q}$ stability of the Ritz projection, we obtain the following estimates: 
\begin{align*}
&\|v_1-R_h v_1\|_{W^{1,q}(\Omega)}\leq C_{q}\inf_{\xh\in \sh(\Omega)}\|v_1-\xh\|_{W^{1,q}(\Omega)}
    \leq C_{q} h \|v_1\|_{W^{2,q}(\Omega)} 
    \leq C_{q} h \|f\|_{L^q(\Omega)}, \\[5pt] 
&\|v_2-R_h v_2\|_{W^{1,q}(\Omega)}
    \leq C_{q}\|v_2\|_{W^{1,q}(\Omega)}
    \leq C_{q} h \|f\|_{L^q(\Omega)} . 
\end{align*}
The result of Lemma \ref{Lemma:Ritz-error-W1p} is obtained by combining the two estimates above. 
\end{proof}

Finally, the $L^p$ error estimate for the Ritz projection follows from a standard duality argument, again by using the regularity decomposition as in \eqref{v-decomp}--\eqref{weak-pde-v2} for the dual problem. 
\begin{lemma}\label{lp-error-ritz} 
Then for any $1<q<2+\varepsilon$ there exists a positive constant $h_q$ such that for $h\le h_q$ the following error estimate holds: 
\begin{align}
    \|u-R_hu\|_{L^{q'}(\Omega)}\leq C_q h\|u-R_hu\|_{W^{1,q'}(\Omega)}
    \quad\forall\, u\in H^1_0(\Omega)\cap W^{1,q'}(\Omega),
\end{align}
where $1/q+1/q'=1$.
\end{lemma}

\begin{proof}
By using the duality between $L^q(\Omega)$ and $L^{q'}(\Omega)$, we can express the $L^{q'}$ error of the Ritz projection as 
$$
\|R_hu-u\|_{L^{q'}(\Omega)}=\sup_{
\begin{subarray}{c}
\varphi\in C^\infty_0(\Omega)\\
\|\varphi\|_{L^{q}(\Omega)}\leq 1
\end{subarray}
}
(R_hu-u, \varphi),
$$
In particular, there exists $\varphi\in C^\infty_0(\Omega)$ with $\|\varphi\|_{L^{q}(\Omega)}\leq 1$ such that 
$$
\|R_hu-u\|_{L^{q'}(\Omega)}\leq 2(R_hu-u, \varphi).
$$
Let $v\in H^1_0(\Omega)$ be the weak solution of the following elliptic equation (in the weak form): 
\begin{align*}
  (A_h\nabla v,\nabla\chi)=(\varphi,\chi)
\quad\forall\,\chi\in H^1_0(\Omega).
\end{align*}
Then
\begin{align*}
    (R_hu-u, \varphi)&=(A_h\nabla v, \nabla (R_hu-u))\\
    &=(A_h\nabla(R_hu-u), \nabla v)\\
    &=(A_h\nabla(R_hu-u), \nabla (v-R_h v))\\
    &\leq C\|R_hu-u\|_{W^{1,q'}(\Omega)}\|R_hv-v\|_{W^{1,q}(\Omega)}\\
    &\leq  C_q h\|R_hu-u\|_{W^{1,q'}(\Omega)}\|\varphi\|_{L^{q}(\Omega)} 
    \quad\mbox{(Lemma \ref{Lemma:Ritz-error-W1p} is used here)} \\
    &\leq C_q h\|R_hu-u\|_{W^{1,q'}(\Omega)}.
\end{align*}
This proves the result of Lemma \ref{lp-error-ritz} . 
\end{proof}

\subsection{Estimation of $\rho^{-\frac{N}2} h^{-1} \|\nabla (v- R_hv)\|_{L^1(\Lambda_h)}$}
\label{sec: main proof}

In this subsection, we prove \eqref{to-prove} by utilizing the results established in Sections \ref{section:Reguliarty}--\ref{section:W1p}, where $v$ is the solution of \eqref{weak-PDE-v}. 
This would complete the proof of Theorem \ref{THM1}. 
To this end, we consider a dyadic decomposition of the domain as in the literature; see \cite{Guzman_Leykekhman_Rossman_Schatz_2009,Leykekhman_Li_2021,Schatz-1980}. 

Let $R_0={\rm diam}(\Omega)$ and $d_j=R_02^{-j}$ for $j\ge 0$. We define a sequence of subdomains
$$
D_j=\{x\in\Omega:d_{j+1} \le|x-x_0|\le d_j \} \quad\mbox{for}\,\,\, j\ge0 . 
$$
For each $j$ we denote by $D_j^l$ a subdomain slightly larger than $D_j$, defined by
$$
D_j^l=
D_{j-l}
\cup
\cdots
\cup
D_{j}
\cup
D_{j+1}
\cup
\cdots
\cup
D_{j+l}\quad \mbox{ ($D_i:=\emptyset$ for $i<0$.)}
$$
Let $J=[\ln_2(R_0/2\kappa\rho)]+1$, where $[\ln_2(R_0/2\kappa\rho)]$ denotes the biggest integer not exceeding $\ln_2(R_0/2\kappa\rho)$.  The constant $\kappa>32$ will be determined below, and the generic constant $C$ will be independent on $\kappa$ until it is determined (unless it contains a subscript $\kappa$).
The definition above implies that 
$$
\frac 12\kappa\rho\le d_{J+1}\le \kappa\rho
$$
and
\begin{align}\label{volume-Aj}
{\rm measure}(D_j\cap \Lambda_h) \leq C h d_j^{N-1} .
\end{align}
Note that $v$ is the solution of \eqref{weak-PDE-v}, where $\varphi=0$ outside $S_\rho(x_0)$. Therefore, $\varphi=0$ in $D_j^3$ for $1\le j\le J$. This result will be used below. 

By using the subdomains defined above, we have
\begin{align}\label{remains-1}
&\rho^{-\frac{N}2} h^{-1} \|\nabla (v- R_hv)\|_{L^1(\Lambda_h)}
\nonumber \\
&\le
\rho^{-\frac{N}2} h^{-1}
\bigg(\sum_{j=0}^J
\|\nabla (v- R_hv)\|_{L^1(\Lambda_h\cap D_j)}
+
\|\nabla (v-R_hv )\|_{L^1(\Lambda_h\cap S_{\kappa\rho}(x_0))}
\bigg)
\nonumber \\
&\le
C\rho^{-\frac{N}2} h^{-1} \sum_{j=0}^J h^{\frac12}d_j^{\frac{N-1}{2}}
\|\nabla (v-R_hv )\|_{L^2(\Lambda_h\cap D_j)}
\nonumber \\
&\quad\,
+
C\kappa^{\frac{N-1}{2}}\rho^{-\frac12} h^{-\frac12} \|\nabla (v-R_hv )\|_{L^2(\Lambda_h\cap S_{\kappa\rho}(x_0))},
\end{align}
where the H\"older inequality and \eqref{volume-Aj} are used in the derivation of the last inequality. By choosing $q=2$ in Lemma \ref{Lemma:Ritz-error-W1p} we have 
\begin{align}\label{a-use1}
    \|\nabla(v-R_hv)\|_{L^2(\Omega)}&\leq Ch\|\varphi\|_{L^2(\Omega)}\leq Ch .
\end{align}
Then, substituting \eqref{a-use1} into the last term on the right-hand side of \eqref{remains-1} and using the fact that $\rho\ge h$ (which follows from the definition of $\rho$ in \eqref{Linfty-1}), we obtain
\begin{align}\label{remain-1}
    \rho^{-\frac{N}2} h^{-1} \|\nabla (v- R_hv)\|_{L^1(\Lambda_h)}
    \leq C\rho^{-\frac{N}2} h^{-\frac12} \sum_{j=0}^J d_j^{\frac{N-1}{2}}
\|\nabla (v-R_hv )\|_{L^2(D_j)}+C_\kappa,
\end{align}
where $C_\kappa$ denotes a constant which depends on the parameter $\kappa$.

It remains to estimate $\|\nabla (v-R_hv )\|_{L^2(D_j)}$. To this end, we use the following interior energy estimate for the solution of \eqref{weak-PDE-v}:
\begin{align}\label{eq:interior-energy}
\|v-R_hv \|_{H^1(D_j)}\leq 
C\|v-\check{I}_hv\|_{H^1(D_j^1)}+Cd_j^{-1}\|v-\check{I}_hv\|_{L^2(D_j^1)}+Cd_j^{-1}\|v-R_hv \|_{L^2(D_j^1)} . 
\end{align}
The proof of such interior energy estimate is omitted as it only requires the coefficient matrix $A_h$ to be $L^\infty$ in the perturbed bilinear form in \eqref{perturbed-bilinear}, without additional smoothness, and therefore is the same as the proof for standard finite elements for the Poisson equation.

We use the decomposition $v=v_1+v_2$ in \eqref{v-decomp}--\eqref{weak-pde-v2} with $f=\varphi$ supported in $S_\rho(x_0)$, and consider interpolation error of $v_1$ and $v_2$, respectively. 
First, by applying the result of Lemma \ref{Proposition-A0} and using the fact that $d_j>h$, we have
\begin{align}\label{interior-H2}
\|v_1-\check{I}_hv_1\|_{H^1(D_j^1)}+{d_j}^{-1}\|v_1-\check{I}_hv_1\|_{L^2(D_j^1)}
\le
Ch\|v_1\|_{H^{2}(D_j^2)}
&\le
Ch d_j^{-1+\frac{N}2-\frac{N}{p}}\|v_1\|_{W^{1,p}(\Omega)}
\nonumber \\
\quad\mbox{for}\,\,\, \mbox{$\frac{2N}{N+2}$}<p<2,
\end{align}
where we have used the following inequality in deriving the last inequality:
\begin{align}\label{interior-harmonic}
\|v_1\|_{H^{2}(D_j^2)}\le
Cd_j^{\frac12-\frac{3}{p}}\|v_1\|_{W^{1,p}(\Omega)}
\quad \mbox{for}\,\,\, \mbox{$\frac{2N}{N+2}$}<p<2.
\end{align}
The inequality above follows from Lemma \ref{lemma: Caccipoli} (because $v_1$ is the solution of \eqref{weak-pde-v1} with $f=\varphi=0$ in $D_j^3$), the H\"older inequality and the Sobolev embedding inequality, i.e.,
$$
\begin{aligned}
\|v_1\|_{H^{2}(D_j^2)}
&\le Cd_j^{-2}\|v_1\|_{L^{2}(D_j^3)}\\
&\le Cd_j^{-2+\frac{N}2-\frac{N}{p_*}}\|v_1\|_{L^{p_*}(D_j^3)} &&\mbox{if $p_*>2$}\\
&\le Cd_j^{-1+\frac{N}2-\frac{N}{p}}\|v_1\|_{W^{1,p}(\Omega)}
&&
\mbox{for $\frac{N}{p_*}=\frac{N}p-1$ and $\mbox{$\frac{2N}{N+2}$}<p<2$}\\
&
&&\mbox{so that $p_*>2$ and $W^{1,p}(\Omega)\hookrightarrow L^{p^*}(\Omega)$}
.
\end{aligned}
$$
Here we require $\kappa>32$ to guarantee that $d_{J+5}>\rho$, which is required in the use Lemma \ref{lemma: Caccipoli}. This proves the last inequality in \eqref{interior-H2}.

Next, we consider the interpolation error of $v_2$ by using Lemma \ref{Proposition-A0} and H\"older inequality, i.e.,  
\begin{align}\label{eq:v_2-interpolate}
    \|v_2-\check{I}_hv_2\|_{H^1(D_j^1)}+{d_j}^{-1}\|v_2-\check{I}_hv_2\|_{L^2(D_j^1)}&\leq Cd_j^{\frac{N}{2}-\frac{N}{p_1}}\|v_2\|_{W^{1,p_1}(D_j^2)} \nonumber\\
    &\leq Cd^{\frac{N}{2}-\frac{N}{q_1}}_jh\|v_1\|_{W^{1,q_1}(\Omega)}\nonumber\\
    &\mbox{for some}\,\,\, p_1>N\,\,\,\mbox{and}\,\,\,\frac{N}{q_1}=\frac{N}{p_1}+1 ,
\end{align}
where we have applied Corollary \ref{v2-weak-cappo} in deriving the last inequality. (Here we only need $p_1$ to be slightly bigger than $N$, and therefore the corresponding $q_1$ here can be smaller than $2$, so that we can use H\"older inequality to estimate $\|\varphi\|_{L^{q_1}(S_\rho(x_0))}$ below.) 

By combining \eqref{interior-H2} and \eqref{eq:v_2-interpolate}, we obtain
\begin{align}\label{eq:v-interpolate}
& C\|v-\check{I}_hv\|_{H^1(D_j^1)}+C{d_j}^{-1}\|v-\check{I}_hv\|_{L^2(D_j^1)} \notag\\
&\leq Ch d_j^{-1+\frac{N}2-\frac{N}{p}}\|v_1\|_{W^{1,p}(\Omega)}+Cd^{\frac{N}{2}-\frac{N}{q_1}}_j h\|v_1\|_{W^{1,q_1}(\Omega)}\nonumber\\
 &\leq Ch d_j^{-1+\frac{N}2-\frac{N}{p}}\rho\|\varphi\|_{L^p(S_\rho(x_0))} +Ch d^{\frac{N}{2}-\frac{N}{q_1}}_j \rho\|\varphi\|_{L^{q_1}(S_\rho(x_0))}  \notag\\ 
    &\leq Chd_j^{-1+\frac{N}2-\frac{N}{p}}\rho^{1-\frac{N}{2}+\frac{N}{p}}
    +Chd_j^{\frac{N}{2}-\frac{N}{q_1}} \rho^{1-\frac{N}{2}+\frac{N}{q_1}},
\end{align}
where we have applied Lemma \ref{lemma:w1p-d} to equation \eqref{weak-pde-v1} in the derivation of the second inequality, and used H\"older inequality in the derivation of the last inequality. 

Finally, substituting \eqref{eq:v-interpolate} into \eqref{eq:interior-energy}, we obtain
\begin{align}\label{interior-3}
&d_j^{\frac{N-1}{2}}\|\nabla (v-R_hv )\|_{L^2(D_j)} \notag\\
&\leq
Chd_j^{N-\frac32-\frac{N}{p}}\rho^{1-\frac{N}{2}+\frac{N}{p}}
    +Chd_j^{N-\frac12-\frac{N}{q_1}} \rho^{1-\frac{N}{2}+\frac{N}{q_1}}
+C d_j^{\frac{N-3}{2}}\|v-R_hv \|_{L^2(D_j^1)}\nonumber\\
&\leq Chd_j^{N-\frac32-\frac{N}{p}}\rho^{1-\frac{N}{2}+\frac{N}{p}}
+C d_j^{\frac{N-3}{2}}\|v-R_hv \|_{L^2(D_j^1)},
\end{align}
where we have chosen $p=q_1<2$ and used $d_j\leq C$ in the derivation of the last inequality. 
Here we can make $p$ as close to $2$ as possible so that $p=q'$ satisfies the condition in Lemma \ref{lp-error-ritz} (which will be used in the subsequent analysis). 

Now we substitute \eqref{interior-3} into \eqref{remain-1} and use the result 
$
    \sum_{j=0}^J d_j^{N-\frac32-\frac{N}{p}}\leq C_\kappa \rho^{N-\frac32-\frac{N}{p}} ,
$
we obtain 
\begin{align}\label{sum-dj-grad}
\sum_{j=0}^Jd_j^{\frac{N-1}{2}}\|\nabla (v-R_hv )\|_{L^2(D_j)} 
\le C_\kappa h\rho^{\frac{N-1}{2}} 
+ \sum_{j=0}^JCd_j^{\frac{N-3}{2}}  \|v-R_hv \|_{L^2(D_j^1)},
\end{align}
and therefore
\begin{align}\label{remains-4}
\rho^{-\frac{N}2} h^{-1} \|\nabla (v-R_hv )\|_{L^1(\Lambda_h)}
&\leq C\rho^{-\frac{N}2} h^{-\frac12} \sum_{j=0}^J d_j^{\frac{N-1}{2}}
\|\nabla (v-R_hv )\|_{L^2(D_j)}+C_\kappa  \notag\\
&\leq
C_\kappa+C\rho^{-\frac{N}2} h^{-\frac12}\sum_{j=0}^Jd_j^{\frac{N-3}{2}} \|v-R_hv \|_{L^2(D_j^1)}
. 
\end{align}

It remains to estimate $\sum_{j=0}^Jd_j^{\frac{N-3}{2}}\|v-R_hv \|_{L^2(D_j^1)}$. To this end, we let $\chi$ be a smooth cut-off function satisfying
$$
\chi=1
\,\,\,\mbox{on}\,\,\,
D_j^1
,\quad
\chi=0
\,\,\,\mbox{outside}\,\,\,
D_j^2 
\quad\mbox{and}\quad |\nabla\chi|\le Cd_j^{-1} . 
$$

For $N=2,3$ the following Sobolev interpolation inequality holds:
\begin{align}\label{theta-p-N-1}
\|\chi(v-R_hv) \|_{L^2(\Omega)}
\le
\|\chi(v-R_hv) \|_{L^p(\Omega)}^{1-\theta}
\|\chi(v-R_hv) \|_{H^1(\Omega)}^\theta
\quad\mbox{with}\,\,\,
\frac{1}{2}
=\frac{1-\theta}{p}+\frac{\theta}{p_*} ,
\end{align}
where $p_*=\infty$ for $N=2$ and $p_*=6$ for $N=3$. 
For both $N=2$ and $N=3$, the parameter $\theta$ determined by \eqref{theta-p-N-1}
satisfies the following relation: 
\begin{align}\label{theta-p-N}
\frac{N}{p}-\frac{N}{2} = \frac{\theta}{1-\theta} . 
\end{align}
We can choose $p$ sufficiently close to $2$ as mentioned below \eqref{interior-3}. 
Since  
\begin{align}\label{v-vh-L6}
C\|\chi(v-R_hv )\|_{H^1(\Omega)} 
&\le
C\|\nabla(v-R_hv )\|_{L^2(D_j^2)}
+Cd_j^{-1}\|v-R_hv \|_{L^2(D_j^2)} 
\end{align} 
it follows that 
\begin{align*}
&\|v-R_hv \|_{L^2(D_j^1)} \\
&\le
\|v-R_hv \|_{L^p(D_j^2)}^{1-\theta}
\big(C\|\nabla(v-R_hv )\|_{L^2(D_j^2)}
+Cd_j^{-1}\|v-R_hv \|_{L^2(D_j^2)}\big)^{\theta} \\
&=
(\epsilon^{-\frac{\theta}{1-\theta}}\|v-R_hv \|_{L^p(D_j^2)})^{1-\theta}
\big(C\epsilon\|\nabla(v-R_hv )\|_{L^2(D_j^2)}
+C\epsilon d_j^{-1}\|v-R_hv \|_{L^2(D_j^2)}\big)^{\theta} \\
&\le
C\epsilon^{-\frac{\theta}{1-\theta}}\|v-R_hv \|_{L^p(D_j^2)}
+C\epsilon\|\nabla(v-R_hv )\|_{L^2(D_j^2)}
+C\epsilon d_j^{-1}\|v-R_hv \|_{L^2(D_j^2)} ,
\end{align*}
where $\epsilon$ can be an arbitrary positive number. 

By choosing $\epsilon=d_j(\rho/d_j)^{\sigma}$ with a fixed $\sigma\in(0,1)$, we obtain
\begin{align}\label{Lp-interpl2}
\|v-R_hv \|_{L^2(D_j^1)}
&\le
C\bigg(\frac{\rho}{d_j}\bigg)^{-\frac{\theta\sigma}{1-\theta}}
d_j^{-\frac{\theta}{1-\theta}}\|v-R_hv \|_{L^p(D_j^1)} \\
&\quad\,
+\bigg(\frac{\rho}{d_j}\bigg)^{\sigma}
\big(Cd_j\|\nabla(v-R_hv )\|_{L^2(D_j^2)}
+C\|v-R_hv \|_{L^2(D_j^2)} \big) .\nonumber
\end{align}
Hence,
\begin{align}\label{Lp-interpl3}
&\rho^{-\frac{N}2} h^{-\frac12}\sum_{j=0}^J d_j^{\frac{N-3}{2}} \|v-R_hv \|_{L^2(D_j^1)} \notag \\
&\le
C\rho^{-\frac{N}2} h^{-\frac12}
\sum_{j=0}^J \bigg(\frac{\rho}{d_j}\bigg)^{-\frac{\theta\sigma}{1-\theta}}
d_j^{-\frac{\theta}{1-\theta}+\frac{N-3}{2}}\|v-R_hv \|_{L^p(D_j^2)} \notag \\
&\quad\,
+
C\rho^{-\frac{N}2} h^{-\frac12}\sum_{j=0}^J
\bigg(\frac{\rho}{d_j}\bigg)^{\sigma}
\big(d_j^{\frac{N-1}{2}}\|\nabla(v-R_hv )\|_{L^2(D_j^2)}
+Cd_j^{\frac{N-3}{2}} \|v-R_hv \|_{L^2(D_j^2)} \big) \notag \\
&\le
C\rho^{-\frac{N}2} h^{-\frac12}\sum_{j=0}^J \bigg(\frac{\rho}{d_j}\bigg)^{-\frac{\theta\sigma}{1-\theta}}
d_j^{-\frac{\theta}{1-\theta}+\frac{N-3}{2}}\|v-R_hv \|_{L^p(D_j^2)}\notag \\
&\quad\,
+
C_\kappa 
+
C\rho^{-\frac{N}2} h^{-\frac12}\bigg(\frac{\rho}{d_J}\bigg)^{\sigma}\sum_{j=0}^J d_j^{\frac{N-3}{2}}  \|v-R_hv \|_{L^2(D_j^{3})} ,
\end{align}
where we have used \eqref{sum-dj-grad} and the fact $\frac{\rho}{d_j}\leq \frac{\rho}{d_J}$ in deriving the last inequality.
Note that
\begin{align*}
\sum_{j=0}^J d_j^{\frac{N-3}{2}}\|v-R_hv \|_{L^2(D_j^{3})} 
&\le
C d_J^{\frac{N-3}{2}}\|v-R_hv \|_{L^2(S_{\kappa\rho}(x_0))}
+
3\sum_{j=0}^J d_j^{\frac{N-3}{2}}\|v-R_hv \|_{L^2(D_j^1)} .
\end{align*}
Combining the last two estimates, we obtain
\begin{align*}
\rho^{-\frac{N}2} h^{-\frac12} \sum_{j=0}^Jd_j^{\frac{N-3}{2}}\|v-R_hv \|_{L^2(D_j^1)} 
&\le
C\rho^{-\frac{N}2} h^{-\frac12}\sum_{j=0}^J \bigg(\frac{\rho}{d_j}\bigg)^{-\frac{\theta\sigma}{1-\theta}}
d_j^{-\frac{\theta}{1-\theta}+\frac{N-3}{2}}\|v-R_hv \|_{L^p(D_j^2)}
\notag \\
&\quad\,
+
C_\kappa 
+ C\rho^{-\frac{N}2} h^{-\frac12}\bigg(\frac{\rho}{d_J}\bigg)^{\sigma}d_J^{\frac{N-3}{2}}\|v-R_hv \|_{L^2(S_{\kappa\rho}(x_0))} \\
&\quad\,
+
C\rho^{-\frac{N}2} h^{-\frac12}\bigg(\frac{\rho}{d_J}\bigg)^{\sigma}\sum_{j=0}^J d_j^{\frac{N-3}{2}}\|v-R_hv \|_{L^2(D_j^1)}.\\
\end{align*}
For the fixed $\sigma\in(0,1)$, by choosing a sufficiently large parameter $\kappa$ we have 
$
  \big(\frac{\rho}{d_J}\big)^\sigma\leq \frac{C}{\kappa^\sigma},
$
and therefore the last term of the inequality above can be absorbed by the left-hand side.
From now on we fix the parameter $\kappa$. Then we have
\begin{align}\label{L2-Lp}
\sum_{j=0}^J\rho^{-\frac{N}2} h^{-\frac12}d_j^{\frac{N-3}{2}} \|v-R_hv \|_{L^2(D_j^1)} 
&\le
\sum_{j=0}^J C\rho^{-\frac{N}2} h^{-\frac12}
\bigg(\frac{\rho}{d_j}\bigg)^{-\frac{\theta\sigma}{1-\theta}}
d_j^{-\frac{\theta}{1-\theta}+\frac{N-3}{2}}\|v-R_hv \|_{L^p(D_j^2)} \notag \\
&\quad\,
+
C_\kappa
+ C\rho^{-\frac{N}2} h^{-\frac12}\bigg(\frac{\rho}{d_J}\bigg)^{\sigma}d_J^{\frac{N-3}{2}}\|v-R_hv \|_{L^2(S_{\kappa\rho}(x_0))}. 
\end{align}

It remains to estimate $\|v-R_hv \|_{L^p(D_j^1)}$ and $\|v-R_hv \|_{L^2(S_{\kappa\rho}(x_0))}$. This is done by applying Lemma \ref{lp-error-ritz} (with $q'=p$ therein), Lemma \ref{Lemma:Ritz-error-W1p} (with $q=p$ therein) and H\"older's inequality, i.e., 
\begin{align}\label{Lp-interpl5}
&\|v-R_hv \|_{L^p(\Omega)}
\le
Ch^{2} \|\varphi\|_{L^{p}(\Omega)}
\le Ch^2\rho^{\frac{N}p-\frac{N}2} ,\\
&\|v-R_hv \|_{L^2(\Omega)}
\le
Ch^{2} \quad\mbox{(setting $q'=q=2$ in Lemma \ref{lp-error-ritz}  and Lemma \ref{Lemma:Ritz-error-W1p} )}.
\end{align}
Then, substituting these estimates into \eqref{L2-Lp}, we obtain
\begin{align}\label{L2-Lp-2}
\sum_{j=0}^J\rho^{-\frac{N}2} h^{-\frac12}d_j^{\frac{N-3}{2}} \|v-R_hv \|_{L^2(D_j^1)}
&\le
\sum_{j=0}^J C \bigg(\frac{h}{\rho}\bigg)^{\frac{N}2}\bigg(\frac{h}{d_j}\bigg)^{\frac{3-N}{2}}
\bigg(\frac{\rho}{d_j}\bigg)^{\frac{N}p-\frac{N}2-\frac{\theta\sigma}{1-\theta}}
d_j^{\frac{N}p-\frac{N}2-\frac{\theta}{1-\theta}}\notag \\
&\quad\,
+
C_\kappa + C\bigg(\frac{h}{\rho}\bigg)^{\frac{N}2} \bigg(\frac{h}{d_J}\bigg)^{\frac{3-N}2}\bigg(\frac{\rho}{d_J}\bigg)^{\sigma} 
\end{align} 
By choosing $p<2$ to be sufficiently close to $2$ (so that $q'=p$ satisfies the condition of Lemma \ref{lp-error-ritz}) and using the relation $\frac{N}{p}-\frac{N}{2} = \frac{\theta}{1-\theta}$ as shown in \eqref{theta-p-N}, we obtain
\begin{align}
&\sum_{j=0}^J\rho^{-\frac{N}2} h^{-\frac12}d_j^{\frac{N-3}{2}} \|v-R_hv \|_{L^2(D_j^1)}
\le C.
\end{align}
Then, substituting the last inequality into the right-hand side of \eqref{remains-4}, we obtain 
\begin{align*} 
\rho^{-\frac{N}2} h^{-1} \|\nabla (v- R_hv)\|_{L^1(\Lambda_h)}\le C.
\end{align*}
This proves \eqref{to-prove}  for sufficiently small mesh size, say $h\le h_0$. This condition is required when we use Corollary \ref{v2-weak-cappo}, Lemma \ref{Lemma:Ritz-error-W1p} and  Lemma \ref{lp-error-ritz} in this subsection.

In the case $h\ge h_0$, we denote by $\tilde{g}_h\in S_h(\Omega_h)$ the isoparametric finite element function satisfying $\tilde{g}_h=u_h$ on $\partial\Omega_h$ and  $\tilde{g}_h=0$ at the interior nodes  of the domain $\Omega_h$. Then the following estimate holds: 
$$
\|\tilde{g}_h\|_{L^\infty(\Omega_h)}\le C\|u_h\|_{L^\infty(\partial\Omega_h)} .
$$
Since $\chi_h=u_h-\tilde{g}_h\in \mathring S_h(\Omega_h)$, 
it follows from \eqref{discrete-harmonic} that 
$$
 0 = \int_{\Omega_h}\nabla u_h\cdot \nabla (u_h-\tilde{g}_h) 
    = \|\nabla (u_h-\tilde{g}_h)\|^2_{L^2(\Omega_h)} + \int_{\Omega_h}\nabla \tilde g_h\cdot \nabla (u_h-\tilde{g}_h) ,
$$
and therefore 
$$
\begin{aligned}
\|\nabla (u_h-\tilde{g}_h)\|^2_{L^2(\Omega_h)}
= - \int_{\Omega_h}\nabla \tilde g_h\cdot \nabla (u_h-\tilde{g}_h) 
&\le
C\|\nabla\tilde{g}_h\|_{L^2(\Omega_h)}\|\nabla(u_h-\tilde{g}_h)\|_{L^2(\Omega_h)}.
\end{aligned}
$$
Thus, by using the inverse inequality and the condition $h\ge h_0$, we have
$$
\begin{aligned}
\|\nabla (u_h-\tilde{g}_h)\|_{L^2(\Omega_h)}
\le
C\|\nabla \tilde{g}_h\|_{L^2(\Omega_h)}
\le
Ch^{-1}\|\tilde{g}_h\|_{L^2(\Omega_h)}
&\le
Ch_0^{-1}\|\tilde{g}_h\|_{L^\infty(\Omega_h)} \\
&\le Ch_0^{-1}\|u_h\|_{L^\infty(\partial\Omega_h)}.
\end{aligned}
$$
By using the inverse inequality again, we obtain
\begin{align*}
\|u_h-\tilde{g}_h\|_{L^\infty(\Omega_h)}
&\le
Ch^{-\frac{N}2}\|u_h-\tilde{g}_h\|_{L^2(\Omega_h)} \\
&\le
Ch^{-\frac{N}2}\|\nabla (u_h-\tilde{g}_h)\|_{L^2(\Omega_h)} \\
&\le Ch_0^{-\frac{N}2-1}\|u_h\|_{L^\infty(\partial\Omega_h)}
.
\end{align*}
By the triangle inequality, this proves
 $$
 \|u_h\|_{L^\infty(\Omega_h)} \le  \|\tilde{g}_h\|_{L^\infty(\Omega_h)}+\|u_h-\tilde{g}_h\|_{L^\infty(\Omega_h)}\le C \|u_h\|_{L^\infty(\partial\Omega_h)}
 $$
for $h\ge h_0$.

Combining the two cases $h\le h_0$ and $h\ge h_0$, we obtain the result of Theorem \ref{THM1}.
\qed

\section{Proof of Theorem \ref{THM2}}\label{sec: stability Ritz}
\setcounter{equation}{0}

In this section, we adapt Schatz's argument in \cite{Schatz-1980} to the proof of maximum-norm stability of isoparametric finite element solutions of the Poisson equation in the curvilinear polyhedron considered here. The argument is based on the weak maximum principle established in Theorem \ref{THM1} and the following technical result, which asserts that the $W^{1,\infty}$ regularity estimate of the Poisson equation can hold in a family of larger perturbed domains $\Omega^t$, $t\in[0,\delta]$, such that ${\rm dist}(\partial\Omega^t,\partial\Omega)\sim t$ and the $W^{1,\infty}$ estimate is uniformly with respect to $t\in[0,\delta]$.

\begin{remark}
Here we make a remark on the idea of our proof. To prove Theorem \ref{THM2}, we observe that the numerical solution $u_h$ is in fact the Ritz projection of $u^{(h)}\in H^1_0(\Omega_h)$ which is the exact solution of the Poisson equation on $\Omega_h$:
\begin{align*}
	-\Delta u^{(h)} = f \quad\mbox{in}\,\,\,\Omega_h\quad\mbox{($f$ is extended by zero outside $\Omega$)},
\end{align*}
in the sense that
\begin{align*}
	R_h(u^{(h)}\circ \Phi_h^{-1})=u_h\circ \Phi_h^{-1}.
\end{align*}
Using the weak maximum principle established in Theorem \ref{THM1}, one can imitate the proof of \cite[Theorem 5.1]{Leykekhman_Li_2021} to show that there holds $L^\infty$ stability for our Ritz projection $R_h$. It follows that
\begin{align*}
	\|u^{(h)}-u_h\|_{L^\infty(\Omega_h)}\leq C\|u^{(h)}-I_hu^{(h)}\|_{L^\infty(\Omega_h)}.
\end{align*} 
Now we can obtain the result of Theorem \ref{THM2} as long as we establish the estimate
\begin{align*}
	\|u-u^{(h)}\|_{L^\infty(\Omega_h)}\leq Ch^{r+1}\|f\|_{L^p(\Omega)} \quad \mbox{($p>N$)},
\end{align*}
where we have extended $u$ by zero outside $\Omega$. To this end, we consider employing the maximum principle of harmonic functions since $\Delta(u^{(h)}-u)=0$ in $\Omega\cap\Omega_h$. Here technically we introduce larger perturbed domain $\Omega^t$ and solution $u^t$
\begin{align*}
	 -\Delta u^t = f \quad\mbox{in}\,\,\,\Omega^t,
\end{align*}
in the larger perturbed domain $\Omega^t$. Then using maximum principle, we compare $u$ and $u^{(h)}$ with $u^t$ respectively, for example we have
\begin{align*}
	\|u-u^t\|_{L^\infty(\Omega)}\leq \|u^t\|_{L^\infty(\partial\Omega)}\leq Ch^{r+1}\|u^t\|_{W^{1,\infty}(\Omega^t)}.
\end{align*}
This explains the motivation of establishing Proposition \ref{tabular-neighbor}.
\end{remark}

\begin{proposition}\label{tabular-neighbor}
Let $\Omega$ be a curvilinear polyhedron with edge openings smaller than $\pi$, and define   
$$
\Omega(\varepsilon):=\{x\in \mathbb{R}^N\;:\; {\rm dist}(x,\Omega)<\varepsilon\} , 
$$
which is an $\varepsilon$ neighborhood of $\Omega$. 
Then there exist constants $\delta>0$ and $\lambda>0$ and a family of larger bounded domains $\Omega^t$ satisfying 
$$\Omega(\lambda t)\subseteq\Omega^t \subseteq \Omega(\lambda^{-1} t)\quad\forall t\in[0,\delta] , $$
such that the weak solution $u^t\in H^1_0(\Omega^t)$ of the Poisson equation  
\begin{align}\label{eq:u-delta}
    -\Delta u^t = f \quad\mbox{in}\,\,\,\Omega^t, \quad\mbox{with}\,\,\,  f\in L^p(\Omega^t) \,\,\,\mbox{for some $p>N$}, 
\end{align}
satisfies the following estimate: 
\begin{align}\label{uniform-w2p}
    \|u^t\|_{W^{1,\infty}(\Omega^t)}\leq C_p\|f\|_{L^p(\Omega^t)} \quad\mbox{for}\,\,\, t\in[0,\delta] ,  
\end{align}
where $C_p$ is some constant which is independent of $t\in[0,\delta]$. 
%
\end{proposition}
\begin{proof}
In a standard convex polyhedron $\hat\Omega$, the following estimate holds for $p>N$ (cf. \cite[Lemma 2.1]{Li-Sun-MCOM}):
\begin{align}
	\|\nabla w\|_{L^{\infty}(\hat\Omega)} \leq C_{p}\|\nabla \cdot(a \nabla w)\|_{L^{p}(\hat\Omega)}
	\quad\forall\, w\in H^1_0(\hat\Omega) \,\,\,\mbox{such that}\,\,\, \nabla \cdot(a \nabla w)\in L^2(\hat\Omega) .
\end{align}
where $a=(a_{ij})$ is any symmetric positive definite matrix in $W^{1,q}(\hat\Omega)$ with $q>N$, satisfying the following estimate:
\begin{align}\label{coercivity-a}
C^{-1}|\xi|^2\le a\xi\cdot\xi \le C|\xi|^2 .
\end{align}

On the curvilinear polyhedron $\Omega$ considered in this article, by using a partition of unity we can reduce the problem to an open subset of $\Omega$ which is diffeomorphic to a convex polyhedral cone. Therefore, the following result still holds for $p>N$:
\begin{align}\label{grad-w-Linfty}
	\|\nabla w\|_{L^{\infty}(\Omega)} \leq C_{p}\|\nabla \cdot(a \nabla w)\|_{L^{p}(\Omega)}
	\quad\forall\, w\in H^1_0(\Omega)\,\,\,\mbox{such that}\,\,\, \nabla \cdot(a \nabla w)\in L^2(\Omega) .
\end{align} 

If there exists a smooth diffeomorphism $\Psi_t: \Omega \rightarrow \Omega^t$ (smooth uniformly with respect to $t\in [0,\delta]$), then we can pull the Poisson equation on $\Omega^t=\Psi_t(\Omega)$ back to the curvilinear polyhedron $\Omega$ as an elliptic equation with some coefficient matrix $a$ satisfying \eqref{coercivity-a}, and then use the result in \eqref{grad-w-Linfty}. This would prove \eqref{uniform-w2p}.  
If the partial derivatives of the diffeomorphism from $\Omega$ to $\Omega^t$ can be uniformly bounded with respect to $t\in[0,\delta]$, then the constant in  \eqref{uniform-w2p} is independent of $t\in [0,\delta]$. 

It remains to prove the existence of a smooth diffeomorphism $\Psi_t:\Omega\rightarrow \Omega^t=\Psi_t(\Omega)$. 
This is presented in the following lemma. 
\end{proof}

\begin{lemma}\label{flow-diffeo}
Let $\Omega$ be a curvilinear polyherdon. Then there exist constants $\delta>0$ and $\lambda>0$ (which only depend on $\Omega$), and a family of diffeomorphisms $\Psi_t:\mathbb{R}^N\rightarrow \mathbb{R}^N$ for $t\in [0,\delta]$, such that 
\begin{enumerate}
\item
$\Omega(\lambda t)\subseteq \Psi_t(\Omega)\subseteq  \Omega(\lambda^{-1} t)$ for $t\in [0,\delta]$ and some constant $\lambda>0$. 

\item
The partial derivatives of $\Psi_t$ are bounded uniformly with respect to $t\in[0,\delta]$, i.e., 
$$
|\nabla^k\Psi_t(x)|\le C_k\quad \forall x\in \mathbb{R}^N , \,\,\,\forall k\ge 1, 
\,\,\,\mbox{where $C_k$ is independent of $t\in[0,\delta]$.} 
$$

\end{enumerate}
\end{lemma}

\begin{proof}
It is known that any given smooth and compactly supported vector field $X$ on $\mathbb{R}$ induces a flow map 
$$
\Psi:\mathbb{R}\times\mathbb{R}^N\to \mathbb{R}^N\quad (t,x)\mapsto \Phi(t,x),
$$
such that each $\Psi_t=\Psi(t,\cdot):\mathbb{R}^N\to \mathbb{R}^N$ is a diffeomorphism of $\mathbb{R}^N$ for sufficiently small $t$, say $|t|\le\delta$. Moreover, $\Psi_0={\rm Id}$, $\Psi_{t+s}=\Psi_t\circ\Psi_s$ for $t,s\in \mathbb{R}$, and the partial derivatives of $\Psi_t$ are uniformly bounded by constants which only depend on $X$ and $\delta$ (independent of $t$). 

Therefore, in order to prove Lemma \ref{flow-diffeo}, it suffices to construct a compactly supported smooth vector field $X$, such that the flow map induced by $X$ satisfies $\Omega(\lambda t)\subseteq \Psi_t(\Omega)\subseteq  \Omega(\lambda^{-1} t)$ for $t\in [0,\delta]$ (with some constants $\lambda>0$ and $\delta>0$). 
This can be proved by utilizing the following result, which provides a criteria for the construction of such a vector field. 

\begin{lemma}\label{vector-cond}
Let $\Omega$ be a curvilinear polyhedron, and let $X$ be a smooth and compactly supported vector field on $\mathbb{R}^N$ satisfying the following conditions:
\begin{enumerate}
	\item $X|_{\Omega'}\equiv 0$ for some nonempty open subset $\Omega'\subset\subset \Omega$.
	\item $	\langle X(x), N_x\rangle\ge c$ at all smooth points $ x\in \partial\Omega$, 
	where $N_x$ denotes the unit outward normal vector at $x\in\partial\Omega$ and $c>0$ is some constant. 
	
	\item $|X(x)|\le 1\quad\forall x\in \mathbb{R}^N$ 
\end{enumerate}
Then there are constants $\lambda>0$ and $\delta>0$, which only depend on $X$ and $\Omega$, such that the flow map $\Psi_t$ induced by the vector field $X$ has the following property: 
$$\Omega(\lambda t)\subseteq \Psi_t(\Omega)\subseteq  \Omega(\lambda^{-1} t) \quad\! \mbox{for}\,\,\, t\in [0,\delta].$$
\end{lemma}

Let us temporarily assume that Lemma \ref{vector-cond} holds, and use it to prove Lemma \ref{flow-diffeo}. To this end, it suffices to construct a vector field which satisfies the conditions in Lemma \ref{vector-cond}. 

From the definition of the curvilinear polyhedron we know that for every $x\in \partial\Omega$ there exists a map $\varphi_x:U_x\to B_{\varepsilon_x}(0)$ which is a diffeomorphism from a neighborhood $U_x$ of $x$ in $\R^N$ to a ball centered at $0$ with radius $\varepsilon_x$, such that $\varphi_x(x)=0$ and $\varphi_x(U_x\cap \Omega)= K_x\cap B_0(\varepsilon_x)$, where $K_x=\{y\in \R^3:y/|y|\in \Theta\}$ is a cone corresponding to a spherical region $\Theta\subset {\mathbb S}^2$ which is contained in an open half sphere, say ${\mathbb S}^2_+=\{x\in\R^3:|x|=1,\,\,x_3>0\}$. We shall use the following terminology:

\begin{enumerate}
	
	\item By composing $\varphi_x$ with an additional linear transformation if necessary, we can assume that $ \nabla\varphi_{x}(x)=I$ (which holds only at the point $x$ in $U_x$). 
	
	\item 
	If $p$ is a smooth point on $\partial K_x$ (not on the edges or vertex of $\partial K_x$), then we denote by $N_{x,p}$ the unit outward normal vector of $\partial K_x$ at $p$, and define 
	$$
	\hat N_x=\big\{N_{x,p}: \mbox{$p$ in some smooth piece of $\partial K_x$}\big\} 
	$$
	to be the set of all outward unit normal vectors on the smooth faces of $\partial K_x$. 
	When $x$ is a smooth point of $\partial\Omega$, $\hat N_x$ consists of only one vector, i.e., the usual unit normal vector $N_x$. Therefore, the set $\hat N_x$ can be viewed as generalization of normal vector at $x$ when $x$ is not a smooth point. 
	\item 
	Let $y$ be an interior point in the polyhedral cone $K_x$. Then the unit vector $V_x=-y/|y|$ satisfies that $\langle V_x,N_{x,p} \rangle>0$ for all $N_{x,p}\in \hat N_x$. 
%
\end{enumerate}

We will construct a smooth vector field $X$ on $\R^N$ as follows, by using a partition of unity. By the three properties above and the compactness of $\partial\Omega$, there is constant $c>0$ only dependent on $\Omega$ such that for each $x\in \partial\Omega$, there is a unit vector $V_x\in \mathbb{R}^N$ such that 
$$
\langle V_x,N_{x,p} \rangle\ge 2c \quad\forall N_{x,p}\in \hat N_x .
$$
Since the normal vector at a smooth point of $\partial\Omega$ changes continuously in a smooth piece of $\partial\Omega$, one can shrink the neighborhood $U_x$ of $x\in\partial\Omega$ so that   
$$
\langle V_x, N_y \rangle\ge c\,\,\,\mbox{for all smooth points $y\in \partial\Omega \cap U_x$},
$$
where $N_y$ denotes the unit outward normal vector at $y\in \partial\Omega \cap U_x$. 
We define a smooth vector field $X_x$ on $U_x$ by
$$
X_x(y)=V_x \quad\forall y\in U_x , 
$$
and choose a finite covering $\{U_{x_\ell}\}_{1\le \ell\le L}$ of $\partial\Omega$ from these $U_x$, $x\in \partial\Omega$, and a family of smooth cut-off functions $\{\chi_\ell\}_{1\le \ell\le L}$ such that $0\leq \chi_\ell\leq 1$ and 
$$
{\rm supp} (\chi_\ell)\subseteq U_{x_\ell} \;\mbox{ and}\;\sum_{1\leq \ell\leq L}\chi_\ell(x)=1, \quad \forall x\in \partial\Omega . 
$$
Then we denote by $X_{x_\ell}$ the above-mentioned vector field defined on $U_{x_\ell}$, and define 
$$X=\sum_{\ell=1}^L\chi_\ell X_{x_\ell} ,$$  
so that $X$ is a compactly supported smooth vector field such that 
$$
\langle X(y), N_y\rangle=\sum_{\chi_\ell(y)\neq 0}\chi_\ell(y)\langle X_{x_\ell},  N_y\rangle\geq c,
\quad
\mbox{for all smooth point $y\in \partial\Omega$} . 
$$
and clearly $|X(x)|\leq 1,\;\forall x\in \mathbb{R}^N$. 
This proves the existence of a desired vector field $X$, and therefore completes the proof of Proposition \ref{tabular-neighbor}. 
\end{proof}

\begin{proof}[Proof of Lemma \ref{vector-cond}]
For each $x\in \partial\Omega$, let $\varphi_x:U_x\to B_{\varepsilon_x}(0)$ be the map as in the definition of the curvilinear polyhedron. Here we do not require $\varphi_x(U_x)$ to be a ball so that we can assume $U_x$ to be convex. 

By composing $\varphi_x$ with an additional linear transformation if necessary, we can assume that $ \nabla\varphi_{x}(x)=I$ (as in the proof of Lemma \ref{flow-diffeo}). Since $ c\le \langle X(x), N_x \rangle\le 1$ (as a the condition in Lemma \ref{vector-cond}), we can shrink the neighborhood $U_x$ small enough so that 
\begin{align}\label{pre-YN}
\frac{c}{2}\le \langle(\nabla\varphi_{x}(y))^\top X(y), N_{x,p} \rangle\le 2 \quad \forall y\in U_x,\,\,\, p\in \varphi_x(U_x\cap\partial\Omega) = \varphi_x(U_x)\cap \partial K_x,\notag \\
\mbox{$p$ is a smooth point}.
\end{align}
Moreover, since $(\nabla\varphi_{x})^\top=I$ at $x$, we can shrink $U_x$ so that the following equivalence relation holds: 
$$
d(y_1, y_2) \sim d(\varphi_x(y_1),\varphi_x(y_2)) \quad\forall y_1,y_2\in U_x,
$$ 
where $d(\cdot,\cdot)$ denotes the Euclidean distance in $\mathbb{R}^N$. As a result, 
$$
d(y, U_x\cap \Omega) \sim d(\varphi_x(y), \varphi_x(U_x\cap\Omega))\quad\forall y\in U_x . 
$$ 
We can choose a finite covering $\{U_{x_\ell}\}_{1\le \ell\le L}$ of $\partial\Omega$ from these $U_x$. Then there exists a sufficiently small $\delta>0$ such that for any $x\in \partial\Omega$ there exists $1\le \ell\le L$ such that for all $t\in[0,\delta]$, 
$$
\Psi_t(x)\in U_{x_\ell}\,\,\,\mbox{for some $1\le \ell\le L$}  .
$$ 
Moreover, 
\begin{align}\label{equiv-dist1}
d(\Psi_t(x),\Omega)=d(\Psi_t(x),U_\ell\cap \Omega)  
\end{align}
and 
\begin{align}\label{equiv-dist2}
d(\Psi_t(x),\Omega) \sim d(\varphi_{x_\ell}(\Psi_t(x)),\varphi_{x_\ell}(U_\ell\cap \Omega)) . 
\end{align}
Let $Y_\ell=(\nabla\varphi_{x_\ell})^\top X|_{U_\ell} $ be the pushforward vector field under $\varphi_{x_\ell}$, then $\varphi_{x_\ell}(\Psi_t(x))$ is the integral curve of vector field $Y_\ell$, with initial value point $\varphi_{x_\ell}(x)$. From \eqref{pre-YN} we know that 
$$
\frac{c}{2}\le \langle Y_\ell(z), N_{x_\ell,p}\rangle\le 2\quad \forall z\in \varphi_{x_\ell}(U_{x_\ell}), \,\,\,\forall p\in \varphi_{x_\ell}(U_{x_\ell}\cap\partial\Omega) =\varphi_{x_\ell}(U_{x_\ell})\cap\partial K_{x_\ell} ,
$$
which implies that the integral curve $\varphi_{x_\ell}(\Psi_t(x))$ is flowing outside $\varphi_{x_\ell}(U_{x_\ell}\cap\Omega)$, i.e., 
$$
\frac{ct}{2} \le d(\varphi_{x_\ell}(\Psi_t(x)),\varphi_{x_\ell}(U_{x_\ell}\cap\Omega)) \le 2t .
$$
Then, from the equivalence of distance as shown in \eqref{equiv-dist1}--\eqref{equiv-dist2}, we conclude that there exists a constant $\lambda>0$ such that 
$$
2\lambda t\le d(\Psi_t(x),\Omega) \le \frac12\lambda^{-1} t \quad\forall t\in [0,\delta],\,\,\,\forall x\in\partial\Omega . 
$$

We consider the domain $\Omega(\lambda t):=\{x\in \mathbb{R}^N\;:\; {\rm dist}(x,\Omega)<\lambda t\} \supset\Omega$. 
On the one hand, since $X|_{\Omega'}=0$ for some subdomain $\Omega'\subset\subset\Omega$ it follows that $\Psi_t(\Omega)\cap\Omega(\lambda t)\neq \emptyset$. On the other hand, since $d(\Psi_t(x),\Omega)>\lambda t$ for all $x\in \partial\Omega$, the boundaries of $\Psi_t(\Omega)$ and $\Omega(\lambda t)$ are disjoint. It follows that $\Omega(\lambda t)\subseteq \Psi_t(\Omega)$ for $t\in [0,\delta]$. Similarly, one can prove that $\Omega(\lambda^{-1} t) \supset \Psi_t(\Omega)$. This completes the proof of Lemma \ref{vector-cond}. 
\end{proof}

\begin{lemma}\label{Linfty-auxi}
Let $\Omega^t$ be the domain in Proposition \ref{tabular-neighbor}, satisfying $\Omega(\lambda t)\subseteq\Omega^t \subseteq \Omega(\lambda^{-1} t)$ for $ t\in[0,\delta] $, with $\Omega(\lambda t)=\{x\in \mathbb{R}^N\;:\; {\rm dist}(x,\Omega)<\lambda t\} $. 
Suppose that $f\in L^p(\Omega^t)$ for some $p>N$, and $\Omega_h\subset \Omega^t$ for some $t=O(h^{r+1})$ and $h\le h_1$, where $h_1>0$ is some constant. Let $u\in H^1_0(\Omega)$ and $u^{(h)}\in H^1_0(\Omega_h)$ be the weak solutions of the following PDE problems{\rm:}
\begin{align*}
    & -\Delta u = f \quad\mbox{in}\,\,\,\Omega , \\
    &-\Delta u^{(h)} = f \quad\mbox{in}\,\,\,\Omega_h ,
\end{align*}
and extend $u$ and $ u^{(h)}$ by zero to the larger domain $\Omega^t$. 
Then there exists $h_2>0$ such that for $h\le h_2$ the following estimate holds: 
\begin{align}\label{ea:Linfty-auxi}
    \|u-u^{(h)}\|_{L^\infty(\Omega^t)}\leq Ch^{r+1}\|f\|_{L^p(\Omega^t)}
\end{align}
\end{lemma}
\begin{proof}
Since $\max\limits_{x\in\Omega_h}|\Phi_h(x)-x|\leq C_0h^{r+1}$ for some constant $C_0$, it follows that $\Omega_h\subset\Omega(C_0h^{r+1}) \subset\Omega^t$ for $t=C_0 \lambda^{-1}h^{r+1}$. When $h$ is sufficiently small we have $t=C_0 \lambda^{-1}h^{r+1}\le\delta$ and therefore $\Omega^t$ is well defined.  
Let $u^t\in H^1_0(\Omega^t)$ be a weak solution of the Poisson equation
\begin{align*}
    -\Delta u^t = f \quad\mbox{in}\,\,\,\Omega^t.
\end{align*}
Proposition \ref{tabular-neighbor} implies that  
\begin{align}
    \|u^t\|_{W^{1,\infty}(\Omega^t)}\leq C\|f\|_{L^p(\Omega^t)}  . 
\end{align}
Since $u^t-u$ is harmonic in $\Omega\subset\Omega^t$ and $u^t-u^{(h)}$ is harmonic in $\Omega_h\subset\Omega^t$, the maximum principle of the continuous problem implies that 
\begin{align}
    \|u^t-u^{(h)}\|_{L^\infty(\Omega_h)}&\leq \|u^t-u^{(h)}\|_{L^\infty(\partial\Omega_h)}\nonumber\\
    &=\|u^t\|_{L^\infty(\partial\Omega_h)} &&\mbox{(since $u^{(h)}=0$ on $\partial\Omega_h$)} \nonumber\\
    &\leq Ch^{r+1}\|u^t\|_{W^{1,\infty}(\Omega^t)}\nonumber\\
    &
    \leq Ch^{r+1}\|f\|_{L^p(\Omega^t)},
\end{align}
where we have used the fact that ${\rm dist}(x,\partial\Omega^t)\leq 2C_0h^{r+1}$ for $x\in \partial\Omega_h$. 
Therefore,
\begin{align}\label{u_delta_h-u(h)}
     \|u^t-u^{(h)}\|_{L^\infty(\Omega^t)}
     &\leq  \|u^t-u^{(h)}\|_{L^\infty(\Omega_h)}+ \|u^t\|_{L^\infty(\Omega^t\setminus \Omega_h)}\nonumber\\
     &\leq Ch^{r+1}\|f\|_{L^p(\Omega^t)}  + Ch^{r+1}\|u^t\|_{W^{1,\infty}(\Omega^t)} \notag \\
     &\leq Ch^{r+1}\|f\|_{L^p(\Omega^t)} .
\end{align}
The following result can be proved in the same way: 
\begin{align}\label{u_delta_h-u}
    \|u^t-u\|_{L^\infty(\Omega^t)}\leq Ch^{r+1}\|f\|_{L^p(\Omega^t)} .
\end{align}
The result of Lemma \ref{Linfty-auxi} follows from \eqref{u_delta_h-u(h)}--\eqref{u_delta_h-u} and the triangle inequality. 
\end{proof}

In the following, we prove Theorem \ref{THM2} by using the technical result in Proposition \ref{tabular-neighbor}. 

Let $\Omega^t$ be the domain in Proposition \ref{tabular-neighbor}, satisfying $\Omega(\lambda t)\subseteq\Omega^t \subseteq \Omega(\lambda^{-1} t)$ for $ t\in[0,\delta] $, with $\Omega(\lambda t)=\{x\in \mathbb{R}^N\;:\; {\rm dist}(x,\Omega)<\lambda t\} $. 
For the simplicity of notation, we still denote by $ f\in L^p(\Omega^t)$ an extension of $\tilde f\in L^p(\Omega\cup \Omega_h)$ satisfying 
$
\| f\|_{L^p(\Omega^t)} \le C\|\tilde f\|_{L^p(\Omega\cup \Omega_h)} \le C\| f\|_{L^p(\Omega)} .
$ 

Under assumption \ref{Assumption}, the curvilinear polyhedral domain $\Omega$ can be extended to a larger convex polyhedron $\Omega_*$ with a piecewise flat boundary such that $\overline\Omega\subset \Omega_*$ and the triangulation $\K$ can be extended to a quasi-uniform triangulation $\K_*$ on $\Omega_*$
(thus the triangulation in $\Omega_*\backslash\overline\Omega$ is also isoparametric on its boundary $\partial\Omega$).

Let $\tilde u$ be an extension of $u^{(h)}$ such that $\tilde u=u^{(h)}$ on $\Omega_h$ and $\tilde u=0$ in $\Omega_*\backslash\Omega_h$. 
Let $\mathring S_h(\Omega_*)\subset H^1_0(\Omega_*)$ be the $H^1$-conforming isoparametric finite element space on $\Omega_*$ with triangulation $\K_*$.
Let $\tilde u_h\in\ring S_h(\Omega_*)$ be the Ritz projection of $\tilde u$ defined by 
$$
\int_{\Omega_*} \nabla (\tilde u - \tilde u_h)\cdot\nabla \chi_h = 0 
\quad\forall\,\chi_h\in \ring S_h(\Omega_*) . 
$$
Then 
\begin{align}\label{uh-uh-1}
\|u^{(h)}-u_h\|_{L^\infty(\Omega_h)} 
&=\|\tilde u-u_h\|_{L^\infty(\Omega_h)} \notag\\
&\le \|\tilde u-\tilde u_h\|_{L^\infty(\Omega_h)} + \|\tilde u_h - u_h\|_{L^\infty(\Omega_h)} \notag\\
&\le \|\tilde u-\tilde u_h\|_{L^\infty(\Omega_*)} + \|\tilde u_h - u_h\|_{L^\infty(\Omega_h)} ,
\end{align}
where $\|\tilde u-\tilde u_h\|_{L^\infty(\Omega_*)}$ is the error of the Ritz projection of an $H^1$-conforming FEM in a standard convex polyhedron and therefore can be estimated by using the result on a standard convex polyhedron (or using the interior maximum-norm estimate as in \cite[Theorem 5.1]{Schatz-Wahlbin-1977} and \cite[Proof of Theorem 5.1]{Leykekhman_Li_2021}), i.e.,
\begin{align}\label{tildeu-uh-2}
\|\tilde u-\tilde u_h\|_{L^\infty(\Omega_*)} 
&\le C\ell_h\|\tilde u-I_h\tilde u\|_{L^\infty(\Omega_*)} \notag\\
&\le C\ell_h\| u^{(h)} - I_hu^{(h)} \|_{L^\infty(\Omega_h)} \notag\\
&\le C\ell_h\| u -I_hu \|_{L^\infty(\Omega_h)} +C\ell_hh^{r+1}\|f\|_{L^p(\Omega^t)},
\end{align}
where the last inequality uses the triangle inequality and \eqref{ea:Linfty-auxi}, and $I_h\widetilde{u}$ is the interpolation operator associated with the larger triangulation $\mathscr{K}_*$ which extends the interpolation operator $I_h:C(\overline{\Omega_h})\to S_h(\Omega_h)$ associated with $\mathscr{K}$. Since $\tilde u_h - u_h$ is discrete harmonic in $\Omega_h$, i.e.,
$$
\int_{\Omega_h} \nabla(\tilde u_h - u_h)\cdot\nabla \chi_h \d x 
= \int_{\Omega_h} \nabla(\tilde u - u^{(h)})\cdot\nabla \chi_h \d x 
= 0 \quad\forall\,\chi_h\in \mathring S_h(\Omega_h) ,
$$
it follows from Theorem \ref{THM1} that $\tilde u_h - u_h$ satisfies the discrete maximum principle, i.e.,
\begin{align}\label{uh-uh-3}
\|\tilde u_h - u_h\|_{L^\infty(\Omega_h)}
&\le C\|\tilde u_h - u_h\|_{L^\infty(\partial\Omega_h)} \notag\\
&= C\|\tilde u_h\|_{L^\infty(\partial\Omega_h)} \notag\\
&= C\|\tilde u_h-\tilde u\|_{L^\infty(\partial\Omega_h)}\quad\mbox{(since $\tilde{u}|_{\partial\Omega_{h}}=0$)} \notag\\
&\le C\|\widetilde{u}_h-\widetilde{u}\|_{L^\infty(\Omega_*)}.
\end{align} 
Substituting \eqref{tildeu-uh-2} and \eqref{uh-uh-3} into \eqref{uh-uh-1} yields 
\begin{align*}
\|u^{(h)}-u_h\|_{L^\infty(\Omega_h)} 
&\le C\ell_h\| u - I_hu \|_{L^\infty(\Omega_h)} +C\ell_hh^{r+1}\|f\|_{L^p(\Omega^t)} .
\end{align*}
Since $u^{(h)}=u_h=0$ in $\Omega\backslash\Omega_h$, it follows that 
\begin{align*}
\|u^{(h)}-u_h\|_{L^\infty(\Omega)} 
=\|u^{(h)}-u_h\|_{L^\infty(\Omega\cap\Omega_h)} 
&\le C\ell_h\| u - I_hu \|_{L^\infty(\Omega_h)} +C\ell_hh^{r+1}\|f\|_{L^p(\Omega^t)} .
\end{align*}
Then, combining this with \eqref{ea:Linfty-auxi}, we obtain the following error bound: 
\begin{align*}
\|u-u_h\|_{L^\infty(\Omega)} 
&\le C\ell_h\| u - I_hu \|_{L^\infty(\Omega_h)} +C\ell_hh^{r+1}\|f\|_{L^p(\Omega^t)} .
\end{align*}
Finally, we note that
\begin{align*}
	\|u-\check{I}_hu\|_{L^\infty(\Omega)}=&\|u\circ\Phi_h-I_h(u\circ \Phi_h)\|_{L^\infty(\Omega_h)}\\
	\geq& \|u-I_hu\|_{L^\infty(\Omega_h)}-C\|u-u\circ\Phi_h\|_{L^\infty(\Omega_h)}\\
	\geq& \|u-I_hu\|_{L^\infty(\Omega_h)}-C\|u\|_{W^{1,\infty}(\mathbb{R}^d)}\|\Phi_h-{\rm Id}\|_{L^\infty(\Omega_h)}\\
	\geq&\|u-I_hu\|_{L^\infty(\Omega_h)}-Ch^{r+1}\|u\|_{W^{1,\infty}(\mathbb{R}^d)}\\
	\geq& \|u-I_hu\|_{L^\infty(\Omega_h)}-Ch^{r+1}\|f\|_{L^p(\Omega^t)} .
\end{align*}
This proves the result of Theorem \ref{THM2}.
\qed

\section{Conclusion}\label{section:conclusion}

We have proved the weak maximum principle of the isoparametric FEM for the Poisson equation in curvilinear polyhedral domains with edge openings smaller than $\pi$, which include smooth domains and smooth deformations of convex polyhedra. The proof requires using a duality argument for an elliptic equation with some discontinuous coefficients arising from the use of isoparametric finite elements. Hence, the standard $H^2$ elliptic regularity does not hold for the solution of the corresponding  dual problem. We have overcome the difficulty by decomposing the solution into a smooth $H^2$ part and a nonsmooth $W^{1,p}$ part, separately, and replaced the $H^2$ regularity required in a standard duality argument by some $W^{1,p}$ estimates for the nonsmooth part of the solution. 

As an application of the weak maximum principle, we have proved an $L^\infty$-norm best approximation property of the isoparametric FEM for the Poisson equation. 
All the analysis for the Poisson equation in this article can be extended to elliptic equations with $W^{1,\infty}$ coefficients. However, the current analysis does not allow us to extend the results to curvilinear polyhedral domains with edge openings bigger than $\pi$ (smooth deformations of nonconvex polyhedra) or graded mesh in three dimensions. 
These would be the subject of future research. 

There are other approaches to the maximum principle of finite element methods for elliptic equations using non-obtuse meshes, which is restricted to piecewise linear finite elements and Poisson equation with constant coefficients; see \cite{Gao-Qiu-2021}. The approach in the current manuscript is applicable to elliptic equations with $W^{1,\infty}$ coefficients, general quasi-uniform meshes, and high-order finite elements, and therefore requires completely different analysis from the approaches using non-obtuse meshes.


\begin{thebibliography}{99}

%

\bibitem{Apel_Rosch_Sirch_2009}
T. Apel, A. R\"{o}sch, and D. Sirch,
{\em {$L^\infty$}-error estimates on graded meshes with application to optimal control},
SIAM J. Control Optim. 48 (2009), pp. 1771--1796.

\bibitem{Apel_Winkler_Pfefferer_2018}
T. Apel, M. Winkler,  and J. Pfefferer,
{\em Error estimates for the postprocessing approach applied to
              {N}eumann boundary control problems in polyhedral domains},
IMA J. Numer. Anal. 38 (2018), pp. 1984--2025.

\bibitem{Ludmi}
C. B\v{a}cut\v{a}, A. L. Mazzucato, V. Nistor, and L. Zikatanov, 
{\em Interface and mixed boundary value problems on $n$-dimensional polyhedral domains},
Indiana University Mathematics Journal 32.6 (1983), pp. 801--808.

\bibitem{Bakaev-Thomee-Wahlbin-2002}
 N. Y. Bakaev, V. Thom\'ee, and L. B. Wahlbin, 
 {\em Maximum-norm estimates for resolvents of elliptic finite element operators}, 
 Math. Comp. 72 (2002), pp. 1597--1610. 

\bibitem{Behringer_Leykekhman_Vexler_2019}
N. Behringer, D. Leykekhman, and B. Vexler,
{\em Global and local pointwise error estimates for finite element approximations to the Stokes problem on convex polyhedra},
arXiv:1907.06871.

\bibitem{Brandts_Korotov_Krizek_2009}
J. Brandts, S. Korotov, M. K\v{r}\'{\i}\v{z}ek, and J. \v{S}olc, 
{\em On nonobtuse simplicial partitions},
SIAM Rev. 51 (2009), pp. 317--335.

\bibitem{Brenner_Scott}
S. C. Brenner and L. R. Scott,
{\em The Mathematical Theory of FEMs.}
Third edition. Texts in Applied Mathematics, 15. Springer, New York, 2008.

\bibitem{Milman}
C. P. Calder\'on and M. Milman, 
{\em Interpolation of Sobolev Spaces: The Real Method},


\bibitem{Christof_2017}
C. Christof,
{\em {$L^\infty$}-error estimates for the obstacle problem revisited},
Calcolo 54 (2017), pp. 1243--1264.


\bibitem{Ciarlet_1970}
P. G. Ciarlet,
{\em Discrete maximum principle for finite-difference operators},
Aequationes Math. 4 (1970), pp. 338--352.

\bibitem{Ciarlet_2002}
P. G. Ciarlet,
{\em The FEM for elliptic problems},
SIAM.(2002).

\bibitem{ciarlet-1972}
P. G. Ciarlet and P. A. Raviart,
{\em Interpolation theory over curved elements, with applications to finite element methods},
Computer Methods in Applied Mechanics and Engineering 1.2 (1972): 217-249.

\bibitem{CiarletRaviart_1973}
P. G. Ciarlet and P. A. Raviart,
{\em Maximum principle and uniform convergence for the finite  element method},
Comput. Methods Appl. Mech. Engrg. 2 (1973), pp. 17--31.

\bibitem{1992-Dauge}
M. Dauge,
{\em Neumann and mixed problems on curvilinear polyhedra},
Integr. Equat. Oper. Th. 15 (1992), pp. 227--261.

\bibitem{Demlow_2006}
A. Demlow,
{\em Localized pointwise a posteriori error estimates for gradients of piecewise linear finite element approximations to second-order quadilinear elliptic problems}, SIAM J. Numer. Anal. 44 (2006), pp. 494--514.

\bibitem{dorich_23}
B. Dörich, J. Leibold, and B. Maier,
{\em Optimal $ W^{1,\infty} $-estimates for an isoparametric finite element discretization of elliptic boundary value problems}, Preprint, 2023.

\bibitem{2004-Draganescu}
A. Draganescu, T. F. Dupont, and L. R. Scott,
{\em Failure of the discrete maximum principle for an elliptic finite element problem}, Math. Comp., 74 (2004), pp. 1--23.

\bibitem{Frehse_Rannacher_1978}
J. Frehse and R. Rannacher,
{\em Asymptotic {$L^{\infty }$}-error estimates for linear finite element approximations of quasilinear boundary value problems}, SIAM J. Numer. Anal. 15 (1978), pp. 418--431.

\bibitem{Gao-Qiu-2021}
H. Gao and W. Qiu, 
{\em The pointwise stabilities of piecewise linear finite element method on non-obtuse tetrahedral meshes of nonconvex polyhedra}, J. Sci. Comput. 87 (2021), article no. 53.

\bibitem{Giaquinta}
M. Giaquinta and L. Martinazzi, 
{\em An Introduction to The Regularity Theory for Elliptic Systems, Harmonic Maps and Minimal Graphs},
Edizioni della Normale, Pisa, 2012.

\bibitem{Grisvard_1985}
P. Grisvard,
{\em Elliptic problems in nonsmooth domains.}
Monographs and Studies in Mathematics, 24, Pitman (Advanced Publishing Program), Boston, MA, 1985.

	
\bibitem{Guzman_Leykekhman_Rossman_Schatz_2009}
J. Guzm\'{a}n, D. Leykekhman, J. Rossmann, and A. H. Schatz,
{\em H\"{o}lder estimates for {G}reen's functions on convex polyhedral
 domains and their applications to FEMs},
Numer. Math. 112 (2009), pp. 221--243.


\bibitem{Hohn_Mittelmann_1981}
W. H\"{o}hn and H. D. Mittelmann,
{\em Some remarks on the discrete maximum-principle for finite elements of higher order},
Computing 27 (1981), pp. 145--154.


\bibitem{Jerison-Kenig-1995}
D. Jerison and C. E. Kenig,
{\em The inhomogeneous Dirichlet problems in Lipschitz domains},
J. Func. Anal. 130 (1995), pp. 161--219.

\bibitem{Kashiwabara-Kemmochi-1}
T. Kashiwabara and T. Kemmochi,
{\em Stability, analyticity, and maximal regularity for parabolic finite element problems on smooth domains}, Math. Comp. 89 (2020), pp. 1647--1679.

\bibitem{Kashiwabara-Kemmochi-2}
T. Kashiwabara and T. Kemmochi, 
{\em Pointwise error estimates of linear FEM for Neumann boundary value problems in a smooth domain}, 
Numer. Math. 144 (2020), pp. 553--584.

\bibitem{Korotov_Krizek_2001}
S. Korotov and M. K\v{r}\'{\i}\v{z}ek,
{\em Acute type refinements of tetrahedral partitions of polyhedral domains}, SIAM J. Numer. Anal. 39 (2001), pp. 724--733.

\bibitem{Korotov_Krizek_Pekka_2001}
S. Korotov, M. K\v{r}\'{\i}\v{z}ek, and P. Neittaanm\"{a}ki,
{\em Weakened acute type condition for tetrahedral triangulations and the discrete maximum principle}, Math. Comp. 70 (2001), pp. 107--119.

\bibitem{Lenoir_1986} 
M. Lenoir, 
{\em Optimal isoparametric finite elements and error estimates for domains involving curved boundaries}, SIAM J. Numer. Anal. 23 (1986), pp. 562--580.

\bibitem{Leykekhman_Vexler_2016}
D. Leykekhman and B. Vexler,
{\em Finite element pointwise results on convex polyhedral domains}, SIAM J. Numer. Anal. 54 (2016), pp. 561--587.

\bibitem{Leykekhman_Vexler_2016b}
D. Leykekhman and B. Vexler,
{\em Pointwise best approximation results for {G}alerkin finite element solutions of parabolic problems}, SIAM J. Numer. Anal. 54 (2016), pp. 1365--1384.

\bibitem{Leykekhman_Li_2021}
D. Leykekhman and B. Li, 
{\em Weak discrete maximum principle of FEMs in convex polyhedra}, 
Math. Comp. 90 (2021), pp. 1--18.

\bibitem{2019-Li}
B. Li,
{\em Analyticity, maximal regularity and maximum-norm stability of semi-discrete finite element solutions of parabolic equations in nonconvex polyhedra}, Math. Comp. 88 (2019), pp. 1--44.

\bibitem{2021-Li}
B. Li, 
{\em Maximal regularity of multistep fully discrete FEMs for parabolic equations}, 
IMA J. Numer. Anal. (2021), DOI: 10.1093/imanum/drab019

\bibitem{Li-Sun-MCOM}
B. Li and W. Sun, 
{\em Maximal $L^p$ analysis of finite element solutions for parabolic equations with nonsmooth coefficients in convex polyhedra}, 
Math. Comp. 86 (2017), pp. 1071--1102.

\bibitem{Meinder_Vexler_2016b}
D. Meinder and B. Vexler,
{\em Optimal error estimates for fully discrete {G}alerkin
 approximations of semilinear parabolic equations}, ESAIM Math. Model. Numer. Anal. 52 (2018), pp. 2307--2325.

\bibitem{1974-Nitsche-Schatz}
J. A. Nitsche and A. H. Schatz,
{\em Interior estimates for Ritz-Galerkin methods}, Math. Comp. 28 (1974), pp. 937--958.

\bibitem{Ruas Santos_1982}
V. Ruas Santos,
{\em On the strong maximum principle for some piecewise linear finite element approximate problems of nonpositive type}, J. Fac. Sci. Univ. Tokyo Sect. IA Math. 29 (1982), pp. 473--491.

\bibitem{Schatz-1980}
A. H. Schatz,
{\em A weak discrete maximum principle and stability of the FEM in $L_\infty$ on plane polygonal domains. I}, Math. Comp. 34 (1980), pp. 77--91.

\bibitem{Schatz-Wahlbin-1977}
A. H. Schatz and L. B. Wahlbin,
{\em Interior maximum norm estimates for FEMs}, Math. Comp. 31 (1977), pp. 414--442.

\bibitem{Schatz-Wahlbin-1978}
A. H. Schatz and L. B. Wahlbin,
{\em Maximum norm estimates in the FEM on plane polygonal domains. {I}}, Math. Comp. 32 (1978), pp. 73--109.

\bibitem{Schatz-Wahlbin-1982}
A. H. Schatz and L. B. Wahlbin,
{\em On the quasi-optimality in {$L_{\infty }$} of the {$\dot
 H^{1}$}-projection into finite element spaces}, Math. Comp. 38 (1982), pp. 1--22.

\bibitem{ThomeeWahlbin2000}
V. Thom\'{e}e and L. B. Wahlbin, 
{\em Stability and analyticity in maximum-norm for simplicial Lagrange
  finite element semidiscretizations of parabolic equations with Dirichlet
  boundary conditions}, 
Numer. Math. 87 (2000), pp. 373--389.

\bibitem{Vanselow_2001}
R. Vanselow,
{\em About {D}elaunay triangulations and discrete maximum
 principles for the linear conforming {FEM} applied to the           {P}oisson equation},
Appl. Math. 46 (2001), pp. 13--28.

\bibitem{Guillemin}
G. Victor, and A.Pollack, 
{\em Differential topology},
Vol. 370. American Mathematical Soc., 2010.

\bibitem{Wahlbin_1978}
L. B. Wahlbin, 
{\em Maximum norm error estimates in the FEM with isoparametric quadratic elements and numerical integration},
RAIRO. Analyse num\'erique 12.2 (1978), pp. 173--202.

\bibitem{WangZhang_2012}
J. Wang and R. Zhang, 
{\em Maximum principles for {$P1$}-conforming finite element approximations of quasi-linear second order elliptic equations}, SIAM J. Numer. Anal. 50 (2012), pp. 626--642.

\bibitem{Xu_Zikatanov_1999}
J. Xu and L. Zikatanov,
{\em A monotone finite element scheme for convection-diffusion equations}, Math. Comp. 68 (1999), pp. 1429--1446.





\end{thebibliography}
\end{document}